\documentclass[11pt,a4paper,reqno]{amsart}
\usepackage[english]{babel}
\usepackage[T1]{fontenc}
\usepackage{verbatim}
\usepackage{palatino}
\usepackage{amsmath}
\usepackage{mathabx}
\usepackage{amssymb}
\usepackage{amsthm}
\usepackage{amsfonts}
\usepackage{graphicx}
\usepackage{esint}
\usepackage{color}
\usepackage{mathtools}
\usepackage{overpic}

\usepackage[colorlinks = true, citecolor = black]{hyperref}
\author{Tuomas Orponen}
\address{Department of Mathematics and Statistics\\ University of Jyv\"askyl\"a,
	P.O. Box 35 (MaD)\\
	FI-40014 University of Jyv\"askyl\"a\\
	Finland} 
\email{\href{mailto:tuomas.t.orponen@jyu.fi}{tuomas.t.orponen@jyu.fi}}

\title{Asymmetric sum-product theorems for Katz-Tao sets}
\date{\today}
\thanks{T.O. is supported by the European Research Council (ERC) under the European Union’s Horizon Europe research and innovation programme (grant agreement No 101087499)}

\newcommand{\R}{\mathbb{R}}

\newcommand{\N}{\mathbb{N}}

\newcommand{\Z}{\mathbb{Z}}

\newcommand{\Hd}{\dim_{\mathrm{H}}}

\newcommand{\diam}{\operatorname{diam}}

\def\Barint_#1{\mathchoice
          {\mathop{\vrule width 6pt height 3 pt depth -2.5pt
                  \kern -8pt \intop}\nolimits_{#1}}%
          {\mathop{\vrule width 5pt height 3 pt depth -2.6pt
                  \kern -6pt \intop}\nolimits_{#1}}%
          {\mathop{\vrule width 5pt height 3 pt depth -2.6pt
                  \kern -6pt \intop}\nolimits_{#1}}%
          {\mathop{\vrule width 5pt height 3 pt depth -2.6pt
                  \kern -6pt \intop}\nolimits_{#1}}}

\numberwithin{equation}{section}

\theoremstyle{plain}
\newtheorem{thm}{Theorem}
\numberwithin{thm}{section}
\newtheorem*{"thm"}{"Theorem"}

\newtheorem{lemma}[thm]{Lemma}

\newtheorem{ex}[thm]{Example}
\newtheorem{cor}[thm]{Corollary}
\newtheorem{proposition}[thm]{Proposition}

\newtheorem{claim}[thm]{Claim}

\theoremstyle{definition}

\newtheorem{definition}[thm]{Definition}

\newtheorem{notation}[thm]{Notation}

\theoremstyle{remark}
\newtheorem{remark}[thm]{Remark}

\newcommand{\nref}[1]{(\hyperref[#1]{#1})}

\DeclareMathSymbol{\intop}  {\mathop}{mathx}{"B3}

\begin{document}

\begin{abstract} I prove two variants of the $ABC$ sum-product theorem for $\delta$-separated sets $A,B,C \subset [0,1]$ satisfying Katz-Tao spacing conditions. The main novelty is that the cardinality of the sets $B,C$ need not match their non-concentration exponent. The new $ABC$ theorems are sharp under their respective hypotheses, and imply the previous one.  \end{abstract}

\maketitle

\tableofcontents

\section{Introduction}\label{s:intro}

The \emph{$ABC$ sum-product problem} asks for sufficient conditions on three sets $A,B,C \subset \R$ to guarantee that the "size" of $A + cB$, for some $c \in C$, is significantly larger than the "size" of $A$. When $A,B,C$ are finite sets, and size is measured with cardinality $|\cdot|$, the necessary and sufficient conditions for $\max_{c \in C} |A + cB| \gg |A|$  are the following:
\begin{equation}\label{discreteConditions} \max\{|B|,|C|\} \gg 1 \quad \text{and} \quad |B||C| \gg |A|. \end{equation}
More precisely, the Szemer\'edi-Trotter theorem \cite{MR729791} implies that if $|C|$ is larger than an absolute constant (indeed the constant in the Szemer\'edi-Trotter theorem), then 
\begin{displaymath} \max_{c \in C} |A + cB| \gtrsim \min\{\sqrt{|A||B||C|},|A||B|\}. \end{displaymath}
This proves the sufficiency of the conditions \eqref{discreteConditions}. The necessity can be seen by letting $A,B,C$ be suitable arithmetic progressions, see the introduction to \cite{MR4778059}. 

When $A,B,C \subset \R$ are compact (infinite) sets, and "size" is measured by Hausdorff dimension $\Hd$, the $ABC$ problem was resolved in \cite[Theorem 1.6]{2023arXiv230110199O}. This time the necessary and sufficient conditions for $\sup_{c \in C} \Hd (A + cB) > \Hd A$ are
\begin{equation}\label{dimensionConditions} \max\{\Hd B,\Hd C\} > 0 \quad \text{and} \quad \Hd B + \Hd C > \Hd A. \end{equation} 
These conditions are counterparts of \eqref{discreteConditions} for $\Hd$. The motivation in \cite{2023arXiv230110199O} for studying the Hausdorff dimension version of the $ABC$ problem was to make progress towards the \emph{Furstenberg set conjecture}. The conjecture was eventually resolved by Ren and Wang in \cite{2023arXiv230808819R}, but building on \cite{2023arXiv230110199O}, so \cite[Theorem 1.6]{2023arXiv230110199O} remains a component of the full solution.

The proof of \cite[Theorem 1.6]{2023arXiv230110199O} proceeds via a \emph{$\delta$-discretised} statement. This statement concerns $\delta$-separated sets satisfying non-concentration conditions that capture \eqref{dimensionConditions}:
\begin{thm}\label{thm:ABCConjecture}
Let $\alpha \in [0,1)$ and $\beta,\gamma > 0$ satisfy $\beta + \gamma > \alpha$.  Then, there exist $\chi,\delta_{0} \in (0,\tfrac{1}{2}]$ such that the following holds for all $\delta \in 2^{-\N} \cap (0,\delta_{0}]$. Let $A,B,C \subset \delta \Z \cap [0,1]$ be sets satisfying:
\begin{enumerate}
\item[\textup{(A)}] $|A| \leq \delta^{-\alpha}$;
\item[\textup{(B)}] $B$ is a non-empty Frostman $(\delta,\beta,\delta^{-\chi})$-set;
\item[\textup{(C)}] $C$ is a non-empty Frostman $(\delta,\gamma,\delta^{-\chi})$-set. 
\end{enumerate}
Then there exists $c \in C$ such that
\begin{equation}\label{conclusion2} |\{a + cb : (a,b) \in G\}|_{\delta} \geq \delta^{-\chi}|A|, \qquad G \subset A \times B, \, |G| \geq \delta^{\chi}|A||B|. \end{equation}
In particular $\max_{c \in C} |A + cB|_{\delta} \geq \delta^{-\chi}|A|$. 
\end{thm}

\begin{remark} The constant $\chi = \chi(\alpha,\beta,\gamma) > 0$ is bounded from below when the triple $(\alpha,\beta,\gamma)$ ranges in any compact subset of the domain
\begin{equation}\label{def:Omega} \Omega_{\mathrm{ABC}} := \{(\alpha',\beta',\gamma') : \alpha' \in [0,1), \, \beta' \in (0,1] \text{ and } \gamma' \in (\max\{0,\alpha' - \beta'\},1]\}. \end{equation}
One can either see this by tracking the constants in the proof given in \cite[Theorem 1.6]{2023arXiv230110199O}, or by an \emph{a posteriori} compactness argument like in \cite[Remark 6.2]{2023arXiv230808819R}.
\end{remark} 

Here $|H|_{\delta}$ refers to the number of dyadic $\delta$-cubes intersecting $H \subset \R^{d}$ (or the usual $\delta$-covering number). The notion of Frostman $(\delta,s,C)$-sets is the following one:
\begin{definition}[Frostman $(\delta,s,C)$-set]\label{def:Frostman} Let $s \geq 0$, $C > 0$, and $\delta \in 2^{-\N}$. A set $P \subset \R^{d}$ is called a \emph{Frostman $(\delta,s,C)$-set} if 
\begin{displaymath} |P \cap B(x,r)|_{\delta} \leq Cr^{s}|P|_{\delta}, \qquad x \in \R^{d}, \, r \geq \delta. \end{displaymath} \end{definition} 
Roughly speaking, Theorem \ref{thm:ABCConjecture} implies \cite[Theorem 1.6]{2023arXiv230110199O}, because the sets $A,B,C$ (as in \eqref{dimensionConditions}) contain Frostman $(\delta,s)$-sets with exponents $s \in \{\Hd A,\Hd B,\Hd C\} =: \{\alpha,\beta,\gamma\}$. Therefore, the conditions \eqref{dimensionConditions} imply the condition $\beta + \gamma > \alpha$ in Theorem \ref{thm:ABCConjecture}.

\begin{remark} Theorem \ref{thm:ABCConjecture} is \cite[Theorem 1.7]{2023arXiv230110199O}, except that in \cite[Theorem 1.7]{2023arXiv230110199O} it is assumed that $\beta \leq \alpha$. However, the cases $\beta > \alpha$ (and even $\beta = \alpha$) follow from an earlier result of Bourgain \cite[Theorem 3]{Bo2}, or more precisely the refined version of his result established by He \cite[Theorem 1]{MR4148151}.   \end{remark} 

We mentioned that \cite[Theorem 1.6]{2023arXiv230110199O} is the sharp answer to the $ABC$ sum-product problem when "size" is measured by Hausdorff dimension. This however does not imply that Theorem \ref{thm:ABCConjecture} would be the optimal answer to the $ABC$ sum-product problem when "size" is measured by $|\cdot|_{\delta}$. One could e.g. ask if $\max_{c \in C} |A + cB|_{\delta} \gg |A|_{\delta}$ holds under the most straightforward counterpart of \eqref{discreteConditions}:
\begin{equation}\label{discretisedConditions} \max\{|B|_{\delta},|C|_{\delta}\} \gg 1 \quad \text{and} \quad |B|_{\delta}|C|_{\delta} \gg |A|_{\delta} \quad \text{and} \quad |A|_{\delta} \ll \delta^{-1}. \end{equation}
(The last condition is evidently necessary in the $\delta$-discretised problem.) The conditions \eqref{discretisedConditions} are formally weaker than those in Theorem \ref{thm:ABCConjecture}. The conditions in Theorem  \ref{thm:ABCConjecture}(B)-(C) imply $\diam(B) \gtrsim \delta^{\chi/\beta}$ and $\diam(C) \gtrsim \delta^{\chi/\gamma}$. Since $\chi > 0$ is very small, Theorem \ref{thm:ABCConjecture} sheds almost no light in a situation where either $B$ or $C$ has small diameter. In contrast, \eqref{discretisedConditions} holds in many such cases. 

Unfortunately, the conditions \eqref{discreteConditions} are far too weak to yield $|A + BC|_{\delta} \gg |A|$. As a fundamental example, consider a case where $1 \ll |B|_{\delta} \leq |A|_{\delta} \ll \delta^{-1}$, and $A \subset [0,\delta|A|_{\delta}]$ and $B \subset [0,\delta |B|_{\delta}]$. Then $|A + B [0,1]|_{\delta} \lesssim |A|_{\delta}$. This shows that some non-concentration conditions on $A$ or $B$ are necessary for positive non-trivial results.

The following non-concentration condition was introduced by Katz and Tao \cite{KT01}:
\begin{definition}[Katz-Tao $(\delta,s,C)$-set]\label{def:KT} Let $s \geq 0$, $C,\delta > 0$, and $\delta \in 2^{-\N}$. A set $P \subset \R^{d}$ is called a \emph{Katz-Tao $(\delta,s,C)$-set} if 
\begin{displaymath} |P \cap B(x,r)|_{\delta} \leq C\left(\frac{r}{\delta} \right)^{s}, \qquad x \in \R^{d}, \, r \geq \delta. \end{displaymath}
A Katz-Tao $(\delta,s,1)$-set is called, in brief, a Katz-Tao $(\delta,s)$-set. \end{definition} 

A Katz-Tao $(\delta,s)$-set of cardinality $\geq \delta^{-s}/C$ is a Frostman $(\delta,s,C)$-set, and conversely a Frostman $(\delta,s,C)$-set contains a Katz-Tao $(\delta,s)$-set of cardinality $\gtrsim \delta^{-s}/C$ (see e.g. the proof of \cite[Proposition 3.9]{OrponenIncidenceGeometry}). Using these facts, Theorem \ref{thm:ABCConjecture} can be equivalently formulated in terms of Katz-Tao conditions. Just replace (B)-(C) by the following:
\begin{itemize}
\item[(B')] $B$ is a Katz-Tao $(\delta,\beta)$-set with $|B| \geq \delta^{\chi - \beta}$.
\item[(C')] $C$ is a Katz-Tao $(\delta,\gamma)$-set with $|C| \geq \delta^{\chi - \gamma}$.
\end{itemize}
This shows that Theorem \ref{thm:ABCConjecture} (only) concerns Katz-Tao sets $B$ and $C$ whose cardinality roughly matches their non-concentration exponent. The purpose of this paper is to prove two variants of Theorem \ref{thm:ABCConjecture}, where this is no longer necessary. Here is the first one:

\begin{thm}\label{c3} For every $\alpha \in (0,1)$, $\beta,\gamma \in [\alpha,1]$, and $\eta > 0$ there exist $\delta_{0},\epsilon > 0$ such that the following holds for all $\delta \in 2^{-\N} \cap (0,\delta_{0}]$. 

Let $A,B,C \subset \delta \Z \cap [0,1]$ be sets satisfying the following properties:
\begin{itemize}
\item[(A)] $A$ is a Katz-Tao $(\delta,\alpha)$-set;
\item[(B)] $B$ is a Katz-Tao $(\delta,\beta)$-set;
\item[(C)] $C$ is a Katz-Tao $(\delta,\gamma)$-set;
\item[($\Pi$)] $|B|^{\gamma}|C|^{\beta}\delta^{\beta\gamma} \geq \delta^{-\eta}$.
\end{itemize} Then there exists $c \in C$ such that 
\begin{equation}\label{expansion} |\{a + cb : (a,b) \in G\}|_{\delta} \geq \delta^{-\epsilon}|A|, \qquad G \subset A \times B, \, |G| \geq \delta^{\epsilon}|A||B|. \end{equation} \end{thm}

We make a few remarks to clarify the statement.

\begin{remark} The hypothesis $\beta,\gamma \in [\alpha,1]$ is not too restrictive. If e.g. $\beta < \alpha \leq \gamma$, note that a Katz-Tao $(\delta,\beta)$-set is automatically a Katz-Tao $(\delta,\alpha)$-set, so Theorem \ref{c3} can be applied with exponents $\alpha,\alpha,\gamma$ instead of $\alpha,\beta,\gamma$. Then \eqref{expansion} holds provided $|B|^{\gamma}|C|^{\alpha}\delta^{\alpha\gamma} \geq \delta^{-\eta}$.
\end{remark}

\begin{remark}\label{rem2} For $\beta,\gamma \geq \alpha$, condition ($\Pi$) implies
\begin{displaymath} |B||C| \geq |B|^{\alpha/\beta}|C|^{\alpha/\gamma} = (|B|^{\gamma}|C|^{\beta})^{\alpha/(\beta \gamma)} \gg \delta^{-\alpha} \geq |A|, \end{displaymath}
which we have earlier identified as a necessary condition. Since ($\Pi$) no longer implies $|B||C| \gg |A|$ for $\beta,\gamma < \alpha$, we see that ($\Pi$) is not sufficient for \eqref{expansion} if $\beta,\gamma < \alpha$.
\end{remark}

\begin{remark}\label{rem3} Besides \eqref{discretisedConditions}, another necessary condition for $|A + BC|_{\delta} \gg |A|$ is the following: $\diam(B) \diam(C) \gg \delta$. To see this, note that if $B \subset [0,r_{B}]$ and $C \subset [0,r_{C}]$ with $r_{B}r_{C} \leq \delta$, then $|(a + cb) - a| \leq \delta$ for all $(a,b,c) \in A \times B \times C$, thus $|A + BC|_{\delta} \leq 2|A|$.

Condition ($\Pi$) is the sharp cardinality condition to imply $\diam(B) \diam(C) \gg \delta$, taking into account the Katz-Tao hypotheses of $B$ and $C$. Indeed, by the Katz-Tao conditions, we have the following (in general sharp) inequality:
\begin{displaymath} \diam(B) \diam(C) \geq |B|^{1/\beta}|C|^{1/\gamma}\delta^{2}. \end{displaymath}
 \end{remark}

\begin{remark}\label{rem5} Theorem \ref{c3} implies Theorem \ref{thm:ABCConjecture}. This is not entirely obvious, since the set $A$ in Theorem \ref{thm:ABCConjecture} has no non-concentration hypothesis. However, "morally" any set $A \subset \delta \Z \cap [0,1]$ with $|A| \leq \delta^{-\alpha}$ is Katz-Tao $(\Delta,\alpha)$ at some higher scale $\Delta \in [\delta,1]$, and we may apply Theorem \ref{c3} at that scale. The details are given in Section \ref{s6}.
\end{remark}

\begin{remark} As we discussed, a caveat of Theorem \ref{thm:ABCConjecture} is that it does not cover any situation where either $B$ or $C$ has small diameter. Theorem \ref{c3} does not suffer from this restriction. For example, when $\alpha = \beta = \gamma$, condition ($\Pi$) is (e.g.) satisfied whenever $|B| \gg \delta^{-\alpha/2}$ and $|C| \gg \delta^{-\alpha/2}$. This allows $B,C$ to have diameter (a bit larger than) $\sqrt{\delta}$, which is optimal in the light of Remark \ref{rem3}.  \end{remark} 

As the second main result, we state a variant of Theorem \ref{c3}, where condition (B) is removed at the cost of imposing a mild Frostman condition on $C$, sometimes known (see \cite{2024arXiv241108871W,2025arXiv250921869W}) as the \emph{two-ends condition}. The Frostman property of $C$ implies $\diam(C) \approx_{\delta} 1$, so the "small diameter" obstruction described in Remark \ref{rem3} does not exist. This is visible in Theorem \ref{c4} as a milder version of condition ($\Pi$).

\begin{thm}\label{c4} 

For every $\alpha \in (0,1)$, $\gamma \in [\alpha,1]$, and $\eta > 0$ there exist $\delta_{0},\epsilon > 0$ such that the following holds for all $\delta \in 2^{-\N} \cap (0,\delta_{0}]$. 

Let $A,B,C \subset \delta \Z \cap [0,1]$ be sets satisfying the following properties:
\begin{itemize}
\item[(A)] $A$ is a Katz-Tao $(\delta,\alpha)$-set;
\item[(C)] $C$ is a Katz-Tao $(\delta,\gamma)$-set which is also a Frostman $(\delta,\eta,\delta^{-\epsilon})$-set;
\item[($\Pi$)] $|B|^{\gamma}|C|^{\alpha}\delta^{\alpha\gamma} \geq \delta^{-\eta}$.
\end{itemize} Then there exists $c \in C$ such that 
\begin{equation}\label{expansion2} |\{a + cb : (a,b) \in G\}|_{\delta} \geq \delta^{-\epsilon}|A|, \qquad G \subset A \times B, \, |G| \geq \delta^{\epsilon}|A||B|. \end{equation} \end{thm}

The condition ($\Pi$) is sharp in the following sense: with all the other hypotheses intact, and even if we additionally assume $B$ to be a Katz-Tao $(\delta,\alpha)$-set, the conclusion \eqref{expansion2} may fail if ($\Pi$) is relaxed to $|B|^{\gamma}|C|^{\alpha}\delta^{\alpha\gamma} \geq \delta^{\eta}$.

\begin{ex}\label{ex1} Let $0 < \eta < \min\{\alpha,\gamma\} \leq \max\{\alpha,\gamma\} \leq 1$ and $\beta \in (0,\alpha - \eta)$. Then, the following holds for $\delta \in 2^{-\N}$ small enough. There exist sets $A,B,C$ satisfying \textup{(A)-(C)} of Theorem \ref{c4}, such that $B$ is also a $(\delta,\alpha)$-Katz-Tao set with $|B| = \delta^{-\beta}$, such that $|B|^{\gamma}|C|^{\alpha}\delta^{\alpha \gamma} \geq \delta^{\eta}$, and
\begin{displaymath} |A + BC|_{\delta} \lesssim |A|. \end{displaymath}

The sets $A,B,C$ we construct are not necessarily subsets of $\delta \Z$, as required Theorem \ref{c4}. This is easy to fix afterwards, by replacing points in $A,B,C$ by nearest elements of $\delta \Z$. 

Let $A_{0},B_{0},C_{0} \subset [0,1]$ be arithmetic progressions (AP) containing $0$, as follows:
\begin{itemize}
\item $B_{0} \subset [0,1]$ is an AP with cardinality $|B_{0}| = \delta^{-\beta}$ and spacing $|B_{0}|^{-1} = \delta^{\beta}$;
\item $A_{0} \subset [0,1]$ is an AP with $|A_{0}| = |B_{0}|^{\alpha/(\alpha - \eta)} > |B_{0}|$ and spacing $|A_{0}|^{-1}$;
\item $C_{0} \subset [0,\tfrac{1}{2}]$ is an AP with $|C_{0}| = |B_{0}|^{\eta/(\alpha - \eta)} = |A_{0}|/|B_{0}|$ and spacing $|C_{0}|^{-1}$.
\end{itemize}
It is easy to check that $|A_{0} + B_{0}C_{0}| \sim |A_{0}|$. Next, fix $\delta > 0$, write $\Delta := \delta|A_{0}|^{1/\alpha} \in [\delta,1]$ (the condition $\beta \leq \alpha - \eta$ ensures that $\Delta \leq 1$), and define $A,B,C$ (tentatively) as follows:
\begin{itemize}
\item $A := \Delta A_{0} \subset [0,\Delta]$;
\item $B := \Delta B_{0} \subset [0,\Delta]$, thus $|B| = |B_{0}| = \delta^{-\beta}$;
\item $C := C_{0} + C_{1}$, where $C_{1} \subset [0,\tfrac{1}{2}(\delta/\Delta)]$ is any $\delta$-separated set. The intervals $c_{0} + [0,\delta/\Delta]$ are disjoint for distinct $c_{0} \in C_{0}$, because $C_{0}$ is $(\delta/\Delta)^{\eta}$-separated, so $|C| = |C_{0}||C_{1}|$.
\end{itemize}
It is easy to check (from the definition of $\Delta$, and since $\Delta \geq \delta|B_{0}|^{1/\alpha}$) that both $A,B$ are Katz-Tao $(\delta,\alpha)$-sets. Fix $a \in A$, $b \in B$, and $c = c_{0} + c_{1} \in C$, where $c_{0} \in C_{0}$ and $c_{1} \in C_{1}$, and observe that
\begin{displaymath} |(a + cb) - (a + c_{0}b)| = |c_{1}b| \leq (\delta/\Delta) \cdot \Delta = \delta. \end{displaymath}
Therefore $A + CB \subset (A + C_{0}B)_{\delta}$, where $(A + C_{0}B)_{\delta}$ stands for the $\delta$-neighbourhood of the set $A + C_{0}B = \Delta(A_{0} + C_{0}B_{0})$. In particular, $|A + CB|_{\delta} \lesssim |A_{0} + C_{0}B_{0}| \sim |A_{0}|$.

We now fix $C_{1} \subset [0,\tfrac{1}{2}(\delta/\Delta)]$ to be any Katz-Tao $(\delta,\gamma)$-set with (maximal) cardinality $|C_{1}| \sim \Delta^{-\gamma} = \delta^{-\gamma}|A_{0}|^{-\gamma/\alpha}$. As a consequence,
\begin{displaymath} |C| \geq |C_{1}| \sim \delta^{-\gamma}|A_{0}|^{-\gamma/\alpha} = \delta^{-\gamma}|B|^{-\gamma/(\alpha - \eta)} \quad \Longrightarrow \quad |B|^{\gamma}|C|^{\alpha}\delta^{\alpha \gamma} \gtrsim \delta^{\eta}. \end{displaymath} 
We leave it to the reader to check that $C$ is $(\delta,\eta)$-Frostman. The reason is that $C_{0}$ is $\eta$-Frostman between scales $1$ and $\delta/\Delta$, and $C_{1}$ is $\gamma$-Frostman (with $\gamma \geq \eta$) between scales $\delta/\Delta$ and $\delta$.

\end{ex}

\subsection{Related work} Besides Theorem \ref{thm:ABCConjecture}, there are many $\delta$-discretised sum-product theorems in the literature, see for example \cite{Bo1,Bo2,MR3361775,MR4283564,He19,2023arXiv230602943M,2025arXiv250102131O,2022arXiv221102277P}. 

Sum-product problems are closely related to incidence problems; for instance the necessary and sufficient conditions in \eqref{discreteConditions} followed from the Szemer\'edi-Trotter incidence bound. Recently, Demeter-Wang \cite{2025arXiv251115899D} and Wang-Wu \cite{2024arXiv241108871W,2025arXiv250921869W} have proved very strong $\delta$-discretised incidence theorems under Katz-Tao hypotheses, so let us discuss their relationship with Theorems \ref{c3} and \ref{c4}.

The main theorem in the most recent work \cite{2025arXiv250921869W} yields the following corollary. Assume that $A,B \subset \delta \Z \cap [0,1]$ are a Katz-Tao $(\delta,\alpha)$-set and a Katz-Tao $(\delta,\beta)$-set, respectively (thus $P := A \times B$ is a Katz-Tao $(\delta,\alpha + \beta)$-set). Assume that $C \subset \delta \Z \cap [0,1]$ is a Katz-Tao $(\delta,\gamma)$-set satisfying the \emph{$2$-ends condition} with $\gamma := \min\{\alpha + \beta,2 - \alpha - \beta\}$. Then
\begin{equation}\label{form40} \max_{c \in C} |A + cB|_{\delta} \gtrapprox |A||B||C|^{1/2} \cdot \delta^{(\alpha + \beta)/2}. \end{equation}
Thus, in \cite{2025arXiv250921869W} the exponent $\gamma$ needs to have a specific relation to $\alpha,\beta$. In this special case, how does the numerology in \eqref{form40} compare to Theorem \ref{c4}? Consider $\alpha + \beta \leq 1$, so $\gamma = \alpha + \beta$. Evidently, \eqref{form40} gives $\max_{c \in C} |A + cB| \gg |A|$ precisely when $|B||C|^{1/2} \delta^{(\alpha + \beta)/2} \gg 1$. Theorem \ref{c4}($\Pi$) requires $|B||C|^{\alpha/(\alpha + \beta)} \delta^{\alpha} \gg 1$. Using $|C| \leq \delta^{-\alpha - \beta}$, one may check that the latter condition is weaker (thus sharper) for $\alpha \leq \beta$, and the former is weaker for $\beta < \alpha$. The explanation is that, for $\beta < \alpha$, any improvement over Theorem \ref{c4} needs to use the Katz-Tao $(\delta,\beta)$-set property of $B$ (which is not assumed in Theorem \ref{c4}). 

I wish to add that the main point of Theorems \ref{c3} and \ref{c4} is not the (sometimes) slightly sharper numerology compared to \cite{2025arXiv250921869W} in the special case $\gamma = \min\{\alpha + \beta,2 - \alpha - \beta\}$, but rather that they give new (sharp) information in cases where $\gamma$ does not match the special exponent. As a concrete example, consider $\alpha = \beta = \gamma = \tfrac{1}{2}$. Then Theorem \ref{c3} says that the sharp requirement for the expansion \eqref{expansion2} is $|B||C| \gg \delta^{-1/2}$ (this is even true if $C$ fails the "two-ends" condition). In this case "$\gamma$" is far from the special exponent $1 = \min\{\alpha + \beta,2 - \alpha - \beta\}$, and the estimate \eqref{form40} gives a fairly poor result: it only guarantees the expansion \eqref{expansion2} if $|B||C|^{1/2} \gg \delta^{-1/2}$. The explanation is, of course, that \eqref{form40} does not  exploit the non-concentration of $C$.

\emph{Added at a revision stage:} shortly after this paper appeared on the \emph{arXiv}, Demeter and O'Regan posted the preprint \cite{2025arXiv251115899D} which generalises \cite{MR4869897}, and contains closely related $\delta$-discretised sum-product estimates. The main result \cite[Theorem 1.3]{2025arXiv251115899D} implies quantitative information of the form \eqref{form40} without assuming that $C$ satisfies the $2$-ends condition. However, \cite[Theorem 1.3]{2025arXiv251115899D} works with the assumption that $\gamma = \min\{\alpha + \beta,2 - \alpha - \beta\}$.

\subsection{Proof and paper outline}\label{s:outline} Section \ref{s2} contains preliminaries. Section \ref{s3} contains an auxiliary $ABC$-type result, Proposition \ref{thm:ABC}, whose proof is based on Theorem \ref{thm:ABCConjecture} combined with a fairly basic multi-scale argument. This auxiliary result is otherwise like Theorem \ref{thm:ABCConjecture}, except that the Frostman hypothesis on $B$ is traded in for a Katz-Tao hypothesis on $A$. The set $C$ is still assumed to be Frostman $(\delta,\gamma)$.
Proposition \ref{thm:ABC} is used to prove Theorems \ref{c3} and \ref{c4}. While Proposition \ref{thm:ABC} is needed to prove Theorem \ref{c4}, its statement is formally superseded by the statement of Theorem \ref{c4} (see Remark \ref{rem4}).

In Section \ref{s4} we prove Theorem \ref{c3}. The main observation is that condition ($\Pi$) implies the existence of scales $\delta \leq \Delta_{1} \ll \Delta_{2} \leq 1$ and numbers $\beta',\gamma'$ with the following properties:
\begin{itemize}
\item[(i)] $|B|_{\delta/\Delta_{1} \to \delta/\Delta_{2}} = (\Delta_{2}/\Delta_{1})^{-\beta'}$ with $\beta' > 0$.
\item[(ii)] $C$ is $\gamma'$-Frostman between scales $\Delta_{2}$ and $\Delta_{1}$ with $\gamma' > 0$.
\item[(iii)] $\beta' + \gamma' > \alpha$.
\end{itemize}
The notation $|H|_{R \to r}$ refers to the "branching" of $H$ between scales $R$ and $r$, see Notation \ref{not1} the precise meaning. The three points above enable us to apply Proposition \ref{thm:ABC} to suitable "pieces" of $A$ and $B$ to get a non-trivial lower bound for $|A + cB|_{\delta \Delta_{j}/\Delta_{j + 1} \to \delta}$. Since $\Delta_{j}/\Delta_{j + 1} \gg 1$, this implies a non-trivial lower bound for $|A + cB|_{\delta}$.

In Section \ref{s5} we indicate the modifications needed in the proof of Theorem \ref{c3} to deduce Theorem \ref{c4}. The main observation is that the condition $\gamma' > 0$ in (ii) above is automatic if $C$ is a Frostman $(\delta,\eta)$-set. Therefore, finding the scales $\Delta_{1},\Delta_{2}$ satisfying (i)-(iii) is formally easier. In fact, their existence is guaranteed by the more relaxed hypothesis $(\Pi)$ in Theorem \ref{c4}, and without the Katz-Tao hypothesis on $B$.

Finally, in Section \ref{s6} we deduce Theorem \ref{thm:ABCConjecture} from Theorem \ref{c3} (Theorem \ref{c4} would work equally well). As mentioned in Remark \ref{rem5}, the main idea is to find a scale $\Delta \in [\delta,1]$ such that $A$ is a Katz-Tao $(\Delta,\alpha)$-set, so Theorem \ref{c3} may be applied at this scale.

We conclude by summarising the logical implications between the various results:
\begin{displaymath} \text{Theorem\,} \ref{thm:ABCConjecture} \quad \Longrightarrow \quad \text{Proposition\,\,} \ref{thm:ABC} \quad \Longrightarrow \quad \text{Theorems\,} \ref{c3} \text{ and } \ref{c4}, \end{displaymath} 
and
\begin{displaymath} \text{Theorem } \ref{c4} \quad \Longrightarrow \quad \text{Proposition } \ref{thm:ABC}, \end{displaymath} 
and
\begin{displaymath} \text{Theorem } \ref{c3} \text{ (or Theorem \ref{c4})} \quad \Longrightarrow \quad \text{Theorem } \ref{thm:ABCConjecture}. \end{displaymath}

\subsection*{Acknowledgements} This project started from discussions with Will O'Regan and Pablo Shmerkin. I'm deeply grateful to both of them for sharing their insights. I'm also grateful to two anonymous reviewers for reading the paper carefully, and for exposing numerous small mistakes and patches of poor writing.

\section{Preliminaries}\label{s2}

For $\delta \in 2^{-\mathbb{N}}$ and $P \subset \mathbb{R}^d$, let $\mathcal{D}_{\delta}(P)$ be the dyadic $\delta$-cubes intersecting $P$. We abbreviate $|P|_\delta:=|\mathcal{D}_{\delta}(P)|$ throughout the paper. For a family $\mathcal{A}$ of subsets of $\R^{d}$ (typically dyadic cubes), we denote the union of the sets by $\cup \mathcal{A}$.

\subsection{Frostman and Katz-Tao sets} Frostman and Katz-Tao sets have already been introduced in Definitions \ref{def:Frostman} and \ref{def:KT}. We discuss some of their properties.

\begin{remark} The following property of Katz-Tao sets is elementary, but fundamental enough to deserve a mention. Assume that $P \subset \R^{d}$ is a Katz-Tao $(\delta,s,C)$-set, $\Delta > 0$, and $x_{0} \in \R^{d}$. Let $T(x) = (x - x_{0})/\Delta$. Then $T(P)$ is a Katz-Tao $(\delta/\Delta,s,C)$-set. Indeed,
\begin{displaymath} |T(P) \cap B(x,r)| = |P \cap B(x_{0} + \Delta x,\Delta r)| \leq  C\left(\frac{r}{\delta/\Delta} \right)^{s}, \quad x \in \R^{d}, \, r \geq \delta/\Delta. \end{displaymath} \end{remark} 

We will need the following lemma; nearly the same statement is \cite[Lemma 1.6]{2024arXiv241108871W}, but we add the short proof for completeness.

\begin{lemma}\label{lemma1} Let $s \geq 0$, $\delta > 0$, and $C \geq 1$. Let $P \subset [0,1]^{d}$ be a Katz-Tao $(\delta,s,C)$-set, and let $\delta \leq \rho \leq 1$. Then $P$ contains a Katz-Tao $(\rho,s,C_{d})$-subset $P'$ of cardinality $|P'| \gtrsim_{d} (\delta/\rho)^{s}|P|/C$, where $C_{d} \geq 1$ is a constant depending only on $d$.  \end{lemma}

\begin{remark} In the proof of Lemma \ref{lemma1}, we use the following version of Hausdorff content:
\begin{displaymath} \mathcal{H}^{s}_{\rho,\infty}(E) := \inf \Big\{ \sum_{i = 1}^{\infty} r_{i}^{s} : E \subset \bigcup_{i = 1}^{\infty} B(x_{i},r_{i}) \text{ and } r_{i} \geq \rho \Big\}, \quad E \subset \R^{d}, \, r > 0, \, s \geq 0. \end{displaymath} 
Evidently $\mathcal{H}^{s}_{\rho_{1},\infty}(E) \geq \mathcal{H}^{s}_{\rho_{2},\infty}(E)$ for $\rho_{1} \geq \rho_{2}$.  \end{remark} 

\begin{proof}[Proof of Lemma \ref{lemma1}] Let $\mu := \delta^{s}\mathcal{H}^{0}|_{P}$. Then $\mu(B(x,r)) \leq Cr^{s}$ for all $x \in \R^{d}$ and $\delta \leq r \leq 1$, and $\mu(\R^{d}) = \delta^{s}|P|$, which easily implies
\begin{displaymath} \mathcal{H}^{s}_{\rho,\infty}(P) \geq \mathcal{H}^{s}_{\delta,\infty}(P) \geq \delta^{s}|P|/C. \end{displaymath}
It follows from \cite[Proposition 3.9]{OrponenIncidenceGeometry} that $P$ contains a Katz-Tao $(\rho,s,C_{d})$-subset $P'$ with $|P'| \gtrsim_{d} \mathcal{H}^{s}_{\rho,\infty}(P)\rho^{-s} \geq (\delta/\rho)^{s}|P|/C$. To be accurate, \cite[Proposition 3.9]{OrponenIncidenceGeometry} states the lower bound $|P'| \gtrsim_{d} \mathcal{H}^{s}_{\infty}(P)\rho^{-s}$, but an inspection of the proof shows that all the (dyadic) cubes considered in the argument have side-length $\geq \rho$. \end{proof}  

A standard exhaustion argument gives the following corollary:

\begin{cor}\label{cor2} In addition to the hypotheses of Lemma \ref{lemma1}, let $c > 0$. Then, there exist disjoint Katz-Tao $(\rho,s,C_{d})$-subsets $P_{1},\ldots,P_{N} \subset P$ satisfying
\begin{itemize}
\item $|P_{j}| \gtrsim_{d} c(\delta/\rho)^{s}|P|/C$ \text{ for all } $1 \leq j \leq N$;
\item $|P \, \setminus \, (P_{1} \cup \ldots \cup P_{N})| \leq c|P|$. 
\end{itemize}
\end{cor} 

\begin{proof} Let $P_{1} \subset P$ be the set provided by one application of Lemma \ref{lemma1}. Assume that disjoint Katz-Tao $(\rho,s,C_{d})$-subsets $P_{1},\ldots,P_{k} \subset P$ have already been defined for some $k \geq 1$, and have cardinality $|P_{j}| \gtrsim_{d} c(\delta/\rho)^{s}|P|/C$. If $|P \, \setminus \, (P_{1} \cup \ldots \cup P_{k})| \leq c|P|$, the proof is completed by setting $N := k$. Otherwise $P' := P \, \setminus \, (P_{1} \cup \ldots \cup P_{k})$ is a Katz-Tao $(\delta,s,C)$-set with cardinality $|P'| \geq c|P|$, so Lemma \ref{lemma1} yields another Katz-Tao $(\rho,s,C_{d})$-subset $P_{k + 1} \subset P'$ with $|P_{k + 1}| \gtrsim_{d} (\delta/\rho)^{s}|P'|/C \geq c(\delta/\rho)^{s}|P|/C$. \end{proof} 

\subsection{Uniform sets and branching functions}\label{s:uniformity} We next recall the concepts of \emph{uniform sets} and their associated \emph{branching functions}.

\begin{definition}[Uniform set]\label{def:uniformity}
Let $m \geq 1$, and let
\begin{displaymath} \delta = \Delta_{m} < \Delta_{m - 1} < \ldots < \Delta_{1} \leq \Delta_{0} = 1 \end{displaymath}
be a sequence of dyadic scales.  We say that a set $P\subset [0,1)^{d}$ is \emph{$\{\Delta_j\}_{j=1}^{m}$-uniform} if there is a sequence $\{N_j\}_{j=1}^{m}$ such that $N_{j} \in 2^{\N}$ and $|P\cap Q|_{\Delta_{j}} = N_j$ for all $j\in \{1,\ldots,m\}$ and all $Q\in\mathcal{D}_{\Delta_{j - 1}}(P)$. 
We also extend this definition to $\mathcal{P}\subset\mathcal{D}_{\delta}([0,1)^{d})$ by applying it to $\cup\mathcal{P}$.
\end{definition}

\begin{notation}\label{not1} If $P \subset [0,1)^{d}$ is $\{\Delta_{j}\}_{j = 1}^{m}$-uniform, and $0 \leq i < j \leq n$, the map $Q \mapsto |P \cap Q|_{\Delta_{j}}$ is constant on $\mathcal{D}_{\Delta_{i}}(P)$. We denote the value of this constant $|P|_{\Delta_{i} \to \Delta_{j}}$ ("the branching of $P$ between scales $\Delta_{i}$ and $\Delta_{j}$"). \end{notation} 

For uniform sets, the Frostman $(\delta,s)$-set property implies a Frostman $(\Delta,s)$-set property for $\Delta \in [\delta,1]$. The lemma below is \cite[Lemma 2.17]{2023arXiv230110199O}:

\begin{lemma}\label{OSlemma} Let $\delta \in 2^{-\N}$, $s \in [0,d]$, and $C > 0$. Let $P \subset [0,1)^{d}$ be a Frostman $(\delta,s,C)$-set. Fix $\Delta \in 2^{-\N} \cap [\delta,1]$, and assume that the map
\begin{displaymath} \mathbf{p} \mapsto |P \cap \mathbf{p}|_{\delta}, \qquad \mathbf{p} \in \mathcal{D}_{\Delta}(P), \end{displaymath}
is constant. Then $P$ is a Frostman $(\Delta,s,O_{d}(1)C)$-set.

In particular: if $P \subset [0,1)^{d}$ is $\{2^{-jT}\}_{j = 1}^{m}$-uniform, with $\delta = 2^{-mT}$, then $P$ is a Frostman $(\Delta,s,O_{d}(1)C)$-set for every $\Delta = 2^{-jT}$, for $1 \leq j \leq m$. \end{lemma}

The next lemma produces large uniform subsets, see \cite[Lemma 3.6]{Sh} for the proof.

\begin{lemma} \label{l:uniformization}
Let $P\subset [0,1)^{d}$, $m,T \in \N$, and $\delta := 2^{-mT}$. Let also $\Delta_{j} := 2^{-jT}$ for $0 \leq j \leq m$, so in particular $\delta = \Delta_{m}$. Then, there is a $\{\Delta_j\}_{j=1}^{m}$-uniform set $P'\subset P$ such that
\begin{displaymath}
|P'|_\delta \ge  \left(2T \right)^{-m} |P|_\delta.
\end{displaymath}
In particular, if $\epsilon > 0$ and $T^{-1}\log (2T) \leq \epsilon$, then $|P'|_{\delta} \geq \delta^{\epsilon}|P|_{\delta}$.
\end{lemma}

The next result follows by iterating Lemma \ref{l:uniformization} analogously to the proof of Corollary \ref{cor2}, see \cite[Corollary 7.9]{2023arXiv230110199O} for the details. 

\begin{cor}\label{cor1} For every $\epsilon > 0$, there exists $T_{0} = T_{0}(\epsilon) \geq 1$ such that the following holds for all $\delta = 2^{-mT}$ with $m \geq 1$ and $T \geq T_{0}$. Let $P \subset [0,1]^{d}$ be $\delta$-separated. Then, there exist disjoint $\{2^{-jT}\}_{j = 1}^{m}$-uniform subsets $P_{1},\ldots,P_{N} \subset P$ satisfying
\begin{itemize}
\item $|P_{j}| \geq \delta^{2\epsilon}|P|$ for all $1 \leq j \leq N$;
\item $|P \, \setminus \, (P_{1} \cup \ldots \cup P_{N})| \leq \delta^{\epsilon}|P|$.
\end{itemize}
\end{cor}

There is a very useful $1$-Lipschitz function associated to each uniform set:

\begin{definition}[Branching function]\label{branchingFunction} Let $T \in \N$, let $P \subset [0,1)^{d}$ be a $\{2^{-jT}\}_{j = 1}^{m}$-uniform set, and let $\{N_{j}\}_{j = 1}^{m} \subset \{1,\ldots,2^{dT}\}^{m}$ be the associated sequence. We define the \emph{branching function} $f \colon [0,m] \to [0,dm]$ by setting $f(0) = 0$, and
		\begin{displaymath} f(j) := \frac{\log |P|_{2^{-jT}}}{T} = \frac{1}{T} \sum_{i = 1}^{j} \log N_{i}, \qquad j \in \{1,\ldots,m\}, \end{displaymath}
		and then interpolating linearly. \end{definition}

\subsection{Background on Lipschitz functions}

\begin{definition}[$\epsilon$-linear and superlinear functions] \label{def:eps-linear}
Given a function $f:[a,b]\to\R$ and numbers $\epsilon >0,\sigma \in \R$, we say that \emph{$f$ is $(\sigma,\epsilon)$-superlinear on $[a,b]$} if
\[
f(x) \ge f(a)+\sigma(x-a)-\epsilon(b-a),\qquad x\in [a,b].
\]
We will also use the notation $s_{f}(x,y) := (f(y) - f(x))/(y - x)$ for $x,y \in [a,b]$ with $x < y$.

\end{definition}

The next lemma follows easily from \cite[Lemma 5.21]{SW21}. For full details, see Lemma 2.23 in the first \emph{arXiv} version of \cite{MR4912925}. 

\begin{lemma}\label{l:combinatorial-weak2}
Let $f \colon [0,m] \to [0,dm]$ be a non-decreasing piecewise affine $d$-Lipschitz function with $f(0) = 0$. Then, there exist sequences $0 = a_0 < a_1 <\cdots < a_{n} = m$ and $0 \le \sigma_{1}<\sigma_{2}<\cdots <\sigma_{n} \le d$ with the following properties:
\begin{enumerate}
\item $f$ is $(\sigma_{j + 1},0)$-superlinear on $[a_{j},a_{j + 1}]$ with $\sigma_{j + 1} = s_{f}(a_{j},a_{j + 1})$;
\item $\sum_{j = 0}^{n - 1} (a_{j + 1} - a_{j})\sigma_{j + 1} = f(m)$.
\end{enumerate}
\end{lemma}

A caveat of Lemma \ref{l:combinatorial-weak2} is that it gives no lower bound on the length of the intervals $[a_{j},a_{j + 1}]$. There are several more complex versions of the lemma which provide such additional information (and more), starting from \cite[Lemma 4.6]{Shmerkin22}, see also \cite[Lemma 2.11]{MR4869897}. The version we need here is \cite[Lemma 2.10]{2023arXiv230110199O}. 

\begin{lemma} \label{l:combinatorial-weak}
For every $\epsilon>0$ there is $\tau=\tau(\epsilon)>0$ such that the following holds: for any non-decreasing $1$-Lipschitz function $f:[0,m]\to\R$ with $f(0)=0$ there exist  sequences
\begin{align*}
0&=a_0 < a_1 <\cdots < a_{n} = m,\\
0&\le \sigma_1 <\sigma_2<\cdots <\sigma_{n} \le 1,
\end{align*}
such that:
\begin{enumerate}
   \item $a_{j+1}-a_j \ge \tau m$;
  \item $f$ is $(\sigma_{j + 1},0)$-superlinear on $[a_{j},a_{j + 1}]$, in particular $\sigma_{j + 1} \leq s_{f}(a_{j},a_{j + 1})$;
  \item $\sum_{j=0}^{n - 1} (a_{j+1}-a_j)\sigma_{j + 1} \ge f(m)-\epsilon m$.
\end{enumerate}
\end{lemma}

\begin{remark} The price to pay for (1) (as opposed to Lemma \ref{l:combinatorial-weak2}) is that one loses an "$\epsilon$" in (3), and in (2) one has no \emph{a priori} lower control for $\sigma_{j + 1}$ in terms of $s_{f}(a_{j},a_{j + 1})$. One can gain such lower control by trading in $(\sigma_{j + 1},0)$-superlinearity for $(\sigma_{j + 1},\epsilon)$-superlinearity; this is a key novelty of Demeter and Wang's version \cite[Lemma 2.11]{MR4869897}.  \end{remark} 

If the branching function of a uniform set is superlinear on an interval, a "renormalised piece" of the set is Frostman, as quantified in \cite[Lemma 8.3]{OS23} (see also \cite[Lemma 2.22]{MR4912925} for the "renormalised" version of the original statement which we actually quote below):

\begin{lemma}\label{lemma2} Let $T \geq 1$, and let $P \subset [0,1)^{d}$ be a $\{2^{-jT}\}_{j = 1}^{m}$-uniform set with branching function $f \colon [0,m] \to [0,dm]$. Let $a,b \in \{0,\ldots,m\}$ with $a < b$, and assume that $f$ is $(\sigma,0)$-superlinear on $[a,b]$. Then the following holds for all $Q \in \mathcal{D}_{2^{-aT}}(P)$. 

Write $\Delta := 2^{-T(b - a)}$. Let $S_{Q} \colon \R^{d} \to \R^{d}$ be the rescaling map taking $Q$ to $[0,1)^{d}$. Then $S_{Q}(P \cap Q)$ is a Frostman $(\Delta,\sigma,O_{d,T}(1))$-set. \end{lemma} 

The Katz-Tao condition can also be read off from the behaviour of the branching function (we do not need a "renormalised" version in this case, so the statement is simpler):

\begin{lemma}\label{lemma3} Let $T \geq 2$, $\epsilon \geq 0$, and let $P \subset [0,1)^{d}$ be a $\{2^{-jT}\}_{j = 1}^{m}$-uniform set with branching function $f \colon [0,m] \to [0,dm]$. Assume that $f(m) - f(x) \leq \sigma(m - x) + \epsilon m$ for all $x \in [0,m]$. Then $P$ is a Katz-Tao $(\delta,\sigma,O_{d,T}(\delta^{-\epsilon}))$-set with $\delta := 2^{-mT}$. \end{lemma} 

\begin{proof} Fix $j \in \{0,\ldots,m\}$ and $Q \in \mathcal{D}_{2^{-jT}}(P)$. Write $\Delta := 2^{-jT}$. By uniformity,
\begin{displaymath} |P \cap Q|_{\delta} = |P|_{\delta}/|P|_{\Delta}, \end{displaymath}
where $|P|_{\delta} = 2^{Tf(m)}$, and $|P|_{\Delta} = 2^{Tf(j)} \geq 2^{T(f(m) - \sigma(m - j) - \epsilon m)}$. Therefore $|P \cap Q|_{\delta} \leq 2^{\epsilon mT} \cdot 2^{\sigma T(m - j)} = \delta^{-\epsilon}(\Delta/\delta)^{\sigma}$. To obtain a corresponding estimate for $|P \cap B(x,r)|_{\delta}$, with $x \in \R^{d}$ and $\delta \leq r \leq 1$, cover $B(x,r)$ by $O_{d,T}(1)$ cubes of side $2^{-jT} \sim_{T} r$. 
 \end{proof} 
 
 \begin{remark}\label{rem6} A converse to Lemma \ref{lemma3} also holds, but we will only need the following simple statement. If $\mathcal{P} \subset \mathcal{D}_{\delta}$ is a $\{2^{-jT}\}_{j = 1}^{m}$-uniform Katz-Tao $(\delta,\gamma,1)$-set with $\gamma \geq 0$ and $\delta = 2^{-mT}$, then the branching function $f \colon [0,m] \to [0,dm]$ satisfies $f(m) - f(j) \leq \gamma(m - j)$ for all $j \in \{0,\ldots,m\}$. Indeed, this follows by noting that
 \begin{displaymath} 2^{(f(m) - f(j))T} = \frac{|\mathcal{P}|_{\delta}}{|\mathcal{P}|_{2^{-jT}}} = \frac{|\mathcal{P}|_{\delta}}{|\mathcal{P}|_{\delta}/|\mathcal{P}|_{2^{-jT} \to \delta}} \leq (2^{-jT}/\delta)^{\gamma} = 2^{\gamma(m - j)T}, \end{displaymath} 
 and taking logarithms. \end{remark} 

\subsection{Lemmas from additive combinatorics} We need the asymmetric Balog-Szemer\'edi-Gowers theorem, see the book of Tao and Vu, \cite[Theorem 2.35]{MR2289012}. We state the result in the following slightly weaker form (following Shmerkin's paper \cite[Theorem 3.2]{Sh}):
\begin{thm}[Asymmetric Balog-Szemer\'edi-Gowers theorem]\label{t:BSG} Given $\zeta > 0$, there exists $\epsilon > 0$ such that the following holds for $\delta \in 2^{-\N}$ small enough. Let $A,B \subset \delta \Z \cap [0,1]$ be finite sets, and assume that there exist $c \in [\delta^{\epsilon},1]$ and $G \subset A \times B$ satisfying 
\begin{equation}\label{bsg1} |G| \geq \delta^{\epsilon}|A||B| \quad \text{and} \quad |\{x + cy : (x,y) \in G\}|_{\delta} \leq \delta^{-\epsilon}|A|. \end{equation}
Then there exist subsets $A' \subset A$ and $B' \subset B$ with the properties
\begin{equation}\label{bsg2} |A'||B'| \geq \delta^{\zeta}|A||B| \quad \text{and} \quad |A' + cB'|_{\delta} \leq \delta^{-\zeta}|A|. \end{equation}
\end{thm}

\begin{remark} Formally, \cite[Theorem 3.2]{Sh} only covers the case $c = 1$, but the cases $c \in [\delta^{\epsilon},1]$ are easily reduced to the case $c = 1$ by studying the sets $B_{c} := \{[cb]_{\delta} : b \in B\} \subset \delta \Z$ and $G_{c} := \{(a,[cb]_{\delta}) : (a,b) \in G\} \subset A \times B_{c}$, where $[r]_{\delta}$ refers to the element of $\delta \Z$ minimising the distance to $r$. Indeed, $c \in [\delta^{\epsilon},1]$ implies $|B_{c}| \gtrsim \delta^{\epsilon}|B|$ and $|G_{c}| \gtrsim \delta^{\epsilon}|G|$.  \end{remark} 

We will also need the following version of the Pl\"unnecke-Ruzsa inequality:
\begin{lemma}[Pl\"unnecke-Ruzsa inequality]\label{PRIneq} Let $\delta \in 2^{-\N}$, let $A,B_{1},\ldots,B_{n} \subset \R$ be arbitrary sets, and assume that $|A + B_{i}|_{\delta} \leq K_{i}|A|_{\delta}$ for all $1 \leq i \leq n$, and for some constants $K_{i} \geq 1$. Then, there exists a subset $A' \subset A$ with $|A'|_{\delta} \geq \tfrac{1}{2}|A|_{\delta}$ such that
\begin{displaymath} |A' + B_{1} + \ldots + B_{n}|_{\delta} \lesssim_{\epsilon,n} K_{1}\cdots K_{n} |A'|_{\delta}. \end{displaymath}
\end{lemma}
This form of the inequality is due to Ruzsa \cite{MR2314377}. For a more general result, see \cite[Theorem 1.5]{MR2484645}, by Gyarmati-Matolcsi-Ruzsa. To be accurate, these statements are not formulated in terms of $\delta$-covering numbers, but one may consult \cite[Corollary 3.4]{MR4283564} by Guth-Katz-Zahl to see how to handle the reduction.

 Finally, we need the following \cite[Exercise 6.5.12]{MR2289012} in the book of Tao and Vu:

\begin{lemma}\label{TVLemma} Let $A,B \subset \delta \Z$, and assume that $|A + B| \leq K|A|$ for some $K \geq 1$. Then, for every $N \geq 1$, there exists $A' \subset A$ with $|A'| \geq \tfrac{1}{2}|A|$ satisfying $|A' - B| \lesssim_{N} K^{2^{N}/N}|A|^{1 + 1/N}$.
\end{lemma}

\section{An auxiliary proposition}\label{s3}

We start by establishing the following variant of Theorem \ref{thm:ABCConjecture}, where the Frostman property of $B$ is traded in for a Katz-Tao property of $A$.

\begin{proposition}\label{thm:ABC} For every $\alpha \in [0,1)$ and $\beta,\gamma > 0$ with $\gamma + \beta > \alpha$ there exist $\delta_{0},\epsilon > 0$ such that the following holds for all $\delta \in 2^{-\N} \cap (0,\delta_{0}]$. Let $A,B,C \subset \delta \Z \cap [0,1]$ be sets satisfying the following hypotheses:
\begin{enumerate}
\item[(A)] $A$ is a Katz-Tao $(\delta,\alpha,\delta^{-\epsilon})$-set;
\item[(B)] $|B| \geq \delta^{-\beta}$;
\item[(C)] $C$ is a non-empty Frostman $(\delta,\gamma,\delta^{-\epsilon})$-set.
\end{enumerate}
Then, there exists $c \in C$ such that
\begin{equation}\label{form13} |\{a + cb : (a,b) \in G\}|_{\delta} \geq \delta^{-\epsilon}|A|, \qquad G \subset A \times B, \, |G| \geq \delta^{\epsilon}|A||B|. \end{equation}
\end{proposition}

\begin{remark}\label{rem1} The conclusion "...there exists $c \in C$..." may be upgraded to "...there exist $\geq (1 - \delta^{\epsilon})|C|$ points $c \in C$..." The only caveat is that the constant "$\epsilon$" has to be chosen as "$\epsilon/2$" from the original version. This is because if $C$ is $(\delta,\gamma,\delta^{-\epsilon/2})$-Frostman, and $C' \subset C$ has size $|C'| \geq \delta^{\epsilon/2}|C|$, then $C'$ is $(\delta,\gamma,\delta^{-\epsilon})$-Frostman.

Second, just like in Theorem \ref{thm:ABCConjecture}, the constant $\epsilon > 0$ is uniformly bounded from below on compact subsets of the domain $\Omega_{\mathrm{ABC}}$ introduced in \eqref{def:Omega}. 

Third, Proposition \ref{thm:ABC} continues to hold if the hypothesis $C \subset [0,1]$ is relaxed to $C \subset [-1,1]$. Indeed, if $C \subset \delta \Z \cap [-1,1]$ is a non-empty Frostman $(\delta,\gamma,\delta^{-\epsilon})$-set, then either $C \cap [0,1]$ or $C \cap [-1,0]$ is a Frostman $(\delta,\gamma,2\delta^{-\epsilon})$-set. If the former holds, we may apply Proposition \ref{thm:ABC} directly. If the latter holds, we instead apply Proposition \ref{thm:ABC} to the following subsets of $\delta \Z \cap [0,1]$ satisfying the hypotheses of Proposition \ref{thm:ABC}:
\begin{displaymath} A, \quad \bar{B} := \{-b + 1 : b \in B\}, \quad \text{and} \quad \bar{C} := \{-c : c \in C \cap [-1,0]\}. \end{displaymath}
The conclusion is that the exists $c \in C$ such that
\begin{displaymath} |\{a + (-c) \cdot (-b + 1) : (a,b) \in G\}|_{\delta} \geq \delta^{-\epsilon}|A|, \qquad G \subset A \times B, \, |G| \geq \delta^{\epsilon}|A||B|. \end{displaymath} 
This gives \eqref{form13}, because $|\{a + cb : (a,b) \in G\}|_{\delta} = \{a + (-c) \cdot (-b + 1) : (a,b) \in G\}|_{\delta}$.  \end{remark}

\begin{remark}\label{rem4} Proposition \ref{thm:ABC} will become obsolete after Theorem \ref{c4} has been established (but that needs Proposition \ref{thm:ABC}). To see this, first note that Proposition \ref{thm:ABC} can be reduced to the case $\gamma \leq \alpha$ (otherwise take $\gamma := \alpha$, unless $\alpha = 0$ in which case the claim is anyway obvious using $\beta,\gamma > 0$). Now, given $A,B,C$ as in Proposition \ref{thm:ABC}, start by finding a Katz-Tao $(\delta,\gamma)$-subset $C' \subset C$ with $|C'| \gtrsim \delta^{\epsilon - \gamma}$. Then note that $C'$ is also a Katz-Tao $(\delta,\alpha)$-set, and $|B||C'| \geq \delta^{-\beta - \gamma + \epsilon} \geq \delta^{-\alpha - \eta}$ for some $\eta = \eta(\alpha,\beta,\gamma) < \min\{\gamma,\beta + \gamma - \alpha\}$. This means that conditions (C)-($\Pi$) of Theorem \ref{c4} is satisfied (with exponents $\alpha = \gamma$). \end{remark}

\begin{proof}[Proof of Proposition \ref{thm:ABC}] We reduce the proof of Proposition \ref{thm:ABC} to a case where the sets $A,B$ are both $\{2^{-jT}\}_{j = 1}^{m}$-uniform for some $T \in \N$. Let us take for granted that Proposition \ref{thm:ABC} holds under this additional assumption. We assume that $\epsilon > 0$ can be chosen independently of $T$ (so $\epsilon = \epsilon(\alpha,\beta,\gamma) > 0$), but the threshold $\delta_{0} > 0$ may depend on $T$. (The reader may verify from \eqref{form18} that the choice of $\epsilon$ is eventually independent of $T$.)

We then deduce the general case as follows. Fix $\beta' \in (0,\beta)$ close enough to $\beta$ so that $\gamma + \beta' > \alpha$, and apply the (assumed) uniform special case of Proposition \ref{thm:ABC} to produce the constant $\epsilon := \epsilon(\alpha,\beta',\gamma) > 0$. Then, write 
\begin{equation}\label{form17} \epsilon' := \epsilon'(\alpha,\beta,\gamma) := \tfrac{1}{9}\min\{\beta - \beta',\epsilon\}. \end{equation}
We claim that Proposition \ref{thm:ABC} (without any uniformity assumptions) holds with constant "$\epsilon'$". Fix $A,B,C$ as in the statement, thus $|B| \geq \delta^{-\beta}$. Apply Corollary \ref{cor1} with "$4\epsilon'$" to find $T = T(\epsilon') = T(\alpha,\beta,\gamma) \geq 1$ and disjoint $\{2^{-jT}\}_{j = 1}^{m}$-uniform subsets $A_{1},\ldots,A_{M} \subset A$ and $B_{1},\ldots,B_{N} \subset B$ satisfying
\begin{equation}\label{form14a} |A_{i}| \geq \delta^{4\epsilon'}|A| \quad \text{and} \quad |A \, \setminus \, (A_{1} \cup \ldots \cup A_{M})| \leq \delta^{2\epsilon'}|A|, \end{equation}
and
\begin{equation}\label{form14b} |B_{j}| \geq \delta^{4\epsilon'}|B| \quad \text{and} \quad |B \, \setminus \, (B_{1} \cup \ldots \cup B_{N})| \leq \delta^{2\epsilon'}|B|. \end{equation}
(In particular $M,N \leq \delta^{-4\epsilon'}$.) Writing $|B_{j}| =: \delta^{-\beta_{j}}$, it follows from the choice of $\epsilon'$ that $\beta_{j} \geq \beta'$, so the hypotheses of the (assumed) uniform version of Proposition \ref{thm:ABC} with constants $\alpha,\beta',\gamma$ are valid for each triple $A_{i},B_{j},C$. Therefore, we obtain a threshold $\delta_{0} = \delta_{0}(\alpha,\beta',\gamma,T) > 0$, and for each $1 \leq i \leq M$ and $1 \leq j \leq N$ a subset $C_{ij} \subset C$ with $|C_{ij}| \geq (1 - \delta^{\epsilon})|C|$ such that \eqref{form13} holds for each $c \in C_{ij}$, and with $A_{i},B_{j}$ in place of $A,B$:
\begin{equation}\label{form15} |\{a + cb : (a,b) \in G\}|_{\delta} \stackrel{\eqref{form17}}{\geq} \delta^{-\epsilon}|A_{i}| \geq \delta^{-\epsilon'}|A|, \quad G \subset A_{i} \times B_{j}, \, |G| \geq \delta^{\epsilon}|A_{i}||B_{j}|. \end{equation}
Now it remains to prove the same conclusion for $A,B$ instead of $A_{i},B_{j}$, and with constant $\epsilon'$ instead of $\epsilon$. First, note that 
\begin{displaymath} \sum_{i,j} |C \, \setminus \, C_{ij}| \leq MN\delta^{\epsilon}|C| \stackrel{\eqref{form17}}{\leq} \delta^{9\epsilon' - 8\epsilon'}|C| = \delta^{\epsilon'}|C|, \end{displaymath}
so the set $C' := \bigcap_{i,j} C_{ij}$ satisfies $|C'| \geq (1 - \delta^{\epsilon'})|C|$. Now, fix $c \in C'$, and let $G \subset A \times B$ with $|G| \geq \delta^{\epsilon'}|A||B|$. It follows from \eqref{form14a}-\eqref{form14b} that 
\begin{align*} \sum_{i,j} |G \cap (A_{i} \times B_{j})| & = |G \cap (A_{1} \cup \ldots \cup A_{M}) \times (B_{1} \cup \ldots \cup B_{N})|\\
& \geq \tfrac{1}{2}|G| \geq \tfrac{1}{2}\delta^{\epsilon'}|A||B| \geq \tfrac{1}{2}\delta^{\epsilon'}\sum_{i,j} |A_{i}||B_{j}|, \end{align*}
so there exist $i \in \{1,\ldots,M\}$ and $j \in \{1,\ldots,N\}$ such that $|G \cap (A_{i} \times B_{j})| \geq \tfrac{1}{2}\delta^{\epsilon'}|A_{i}||B_{j}| \geq \delta^{\epsilon}|A_{i}||B_{j}|$. Since $c \in C_{ij}$, we infer from \eqref{form15} that 
\begin{displaymath} |\{a + cb : (a,b) \in G\}|_{\delta} \geq |\{a + cb : (a,b) \in G \cap (A_{i} \times B_{j})\}|_{\delta} \geq \delta^{-\epsilon'}|A|, \end{displaymath}
as desired. This completes the reduction to the case where $A,B$ are $\{2^{-jT}\}_{j = 1}^{m}$-uniform. 

We may now assume that $A,B$ are $\{2^{-jT}\}_{j = 1}^{m}$-uniform for some $T \in \N$. As a much simpler reduction, we may also assume that $C$ is $\{2^{-jT}\}_{j = 1}^{m}$-uniform; this simply follows by applying Lemma \ref{l:uniformization} (for $T \sim_{\epsilon} 1$) to extract a $\{2^{-jT}\}_{j = 1}^{m}$-uniform subset $C'$ with $|C'| \geq \delta^{\epsilon}|C|$, and noting that $C'$ remains $(\delta,\gamma,\delta^{-2\epsilon})$-Frostman.

We proceed to define constants. Fix $\beta' \in (0,\beta)$ such that still $\beta' + \gamma - \alpha > 0$, and write 
\begin{displaymath} \eta := \eta(\alpha,\beta,\gamma) := \tfrac{1}{2}\min\{\beta',\beta - \beta',\beta' + \gamma - \alpha\} > 0. \end{displaymath}
Let $f \colon [0,m] \to [0,m]$ be the branching function of $B$, and let $\{a_{j}\}_{j = 0}^{n}$ and $\{\beta_{j}\}_{j = 1}^{n} := \{\sigma_{j}\}_{j = 1}^{n}$ be the sequences ("scales and slopes") provided by Lemma \ref{l:combinatorial-weak} applied to $f$ and the constant $\eta$. Write $\Delta_{j} := 2^{-a_{j}T}$, so that $\delta = \Delta_{n} < \Delta_{n - 1} < \ldots < \Delta_{0} = 1$. Recall from Lemma \ref{l:combinatorial-weak}(1) that $\Delta_{j + 1}/\Delta_{j} = 2^{-T(a_{j + 1} - a_{j})} \leq \delta^{\tau}$ for $0 \leq j \leq n - 1$, where 
\begin{equation}\label{def:tau} \tau = \tau(\eta) = \tau(\alpha,\beta,\gamma) > 0. \end{equation}
Finally, Lemma \ref{l:combinatorial-weak}(3) tells us that 
\begin{equation}\label{form16} \prod_{j = 0}^{n - 1} \left(\frac{\Delta_{j + 1}}{\Delta_{j}} \right)^{-\beta_{j + 1}} = 2^{T\sum_{j = 0}^{n - 1} (a_{j + 1} - a_{j})\beta_{j + 1}} \geq 2^{T(f(m) - \eta m)} \geq \delta^{-\beta + \eta} \geq \delta^{-\beta'}. \end{equation} 

We are now in a position to describe the constant $\epsilon = \epsilon(\alpha,\beta,\gamma) > 0$ in Proposition \ref{thm:ABC}. Recall $\Omega_{\mathrm{ABC}}$ from \eqref{def:Omega}, and let $K \subset \Omega_{\mathrm{ABC}}$ be the compact set
\begin{equation}\label{def:K} K := \{(\alpha',\beta',\gamma') \in \R^{3} : \alpha' \in [0,\alpha], \, \beta' \in [\eta,1] \text{ and } \gamma' \in [\max\{\eta,\alpha' - \beta' + \eta\},1]\}. \end{equation} 
Let $\chi := \inf \{\chi(\alpha',\beta',\gamma') : (\alpha',\beta',\gamma') \in K\} > 0$ be the constant given by the "classical" $ABC$ theorem, Theorem \ref{thm:ABCConjecture}. Then, let 
\begin{equation}\label{form18} \epsilon := \epsilon(\alpha,\beta,\gamma) := \chi \tau/12, \end{equation} 
where $\tau = \tau(\alpha,\beta,\gamma) > 0$ is defined at \eqref{def:tau}. We will show that \eqref{form13} holds with this "$\epsilon$".

We first claim that there exists an index $j \in \{0,\ldots,n - 1\}$ such that $\beta_{j + 1} \geq \eta$, and
\begin{equation}\label{form4} \left(\frac{\Delta_{j + 1}}{\Delta_{j}} \right)^{-\beta_{j + 1} - \gamma} \geq \left(\frac{\Delta_{j + 1}}{\Delta_{j}} \right)^{-\eta} |A|_{\Delta_{j} \to \Delta_{j + 1}}. \end{equation}
Recall that $|A|_{\Delta_{j} \to \Delta_{j + 1}} = |A \cap Q|_{\Delta_{j + 1}}$ for $Q \in \mathcal{D}_{\Delta_{j}}(A)$. Assume that this fails for each $j \in \{0,\ldots,n - 1\}$. Let $n_{0} := \min\{0 \leq j \leq n - 1 : \beta_{j + 1} \geq \eta\}$ (note that $n_{0}$ is well-defined since $\beta_{n} \geq \beta' \geq \eta$ thanks to \eqref{form16}). Thus \eqref{form4} fails for all $j \in \{n_{0},\ldots,n - 1\}$, and
\begin{align*} \delta^{-\beta' - \gamma} & \leq \prod_{j = 0}^{n - 1} \left(\frac{\Delta_{j + 1}}{\Delta_{j}} \right)^{-\beta_{j + 1} - \gamma}\\
& < \prod_{j = 0}^{n_{0} - 1} \left(\frac{\Delta_{j + 1}}{\Delta_{j}} \right)^{-\eta - \gamma} \prod_{j = n_{0}}^{n - 1} \left(\frac{\Delta_{j + 1}}{\Delta_{j}} \right)^{-\eta} |A|_{\Delta_{j} \to \Delta_{j + 1}}\\
& \leq \delta^{-\eta}\Delta_{n_{0}}^{-\gamma}|A|_{\Delta_{n_{0}} \to \delta}. \end{align*} 
(If $n_{0} = 0$, the first product above is empty, and $\beta_{1} = \beta_{n} \geq \beta'$.) The right hand side is at most $\delta^{-\eta}\Delta_{n_{0}}^{-\gamma}(\Delta_{n_{0}}/\delta)^{\alpha} \leq \delta^{-\beta' - \gamma}$ by the Katz-Tao $(\delta,\alpha)$-set property of $A$, and the choice of $\eta$. This is a contradiction.

Now that the index $j \in \{0,\ldots,n - 1\}$ has been found, we fix it for the remainder of the proof. We claim that we may assume $|A|_{\Delta_{j} \to \Delta_{j + 1}} \leq (\Delta_{j}/\Delta_{j + 1})^{\alpha}$ without loss of generality. Indeed, if the converse holds, then \eqref{form4} implies $\beta_{j + 1} + \gamma \geq \alpha + \eta$. Consequently,
\begin{equation}\label{form46} \left(\frac{\delta}{\Delta_{j}} \right)^{-\beta_{j + 1} - \gamma} \geq \left(\frac{\delta}{\Delta_{j}} \right)^{-\alpha - \eta} \geq \left(\frac{\delta}{\Delta_{j}} \right)^{-\eta}|A|_{\Delta_{j} \to \delta}, \end{equation} 
using the Katz-Tao $(\delta,\alpha)$-set hypothesis of $A$. In other words, the analogue of \eqref{form4} holds with "$\delta$" in place of "$\Delta_{j + 1}$". Furthermore, $|A|_{\Delta_{j} \to \delta} \leq (\Delta_{j}/\delta)^{\alpha}$, once more by the Katz-Tao hypothesis. At this point we may redefine $\Delta_{j + 1} := \delta$. Then, thanks to \eqref{form46}, we have both \eqref{form4} and $|A|_{\Delta_{j} \to \Delta_{j + 1}} \leq (\Delta_{j}/\Delta_{j + 1})^{\alpha}$ simultaneously. These are all the properties of $\Delta_{j},\Delta_{j + 1}$ we will use in the sequel.

Let $I \in \mathcal{D}_{\Delta_{j}}(A)$ and $J \in \mathcal{D}_{\Delta_{j}}(B)$, and define the following subsets of $\mathcal{D}_{\Delta_{j + 1}/\Delta_{j}}(\R)$: 
\begin{displaymath} A_{I} := S_{I}(\mathcal{D}_{\Delta_{j + 1}}(A \cap I)) \quad \text{and} \quad B_{J} := S_{J}(\mathcal{D}_{\Delta_{j + 1}}(B \cap J)), \end{displaymath}
where $S_{I},S_{J} \colon \R \to \R$ are affine maps rescaling $I,J$ to $[0,1)$. Write $\Delta := \Delta_{j + 1}/\Delta_{j}$. Then:
\begin{itemize}
\item $|A_{I}| = |A|_{\Delta_{j} \to \Delta_{j + 1}} =: \Delta^{-\alpha'}$ with $\alpha' \leq \alpha$;
\item $B_{J}$ is $(\Delta,\beta_{j + 1},O_{T}(1))$-Frostman with $\beta_{j + 1} \geq \eta$ due to the $(\beta_{j + 1},0)$-superlinearity of $f$ on $[a_{j},a_{j + 1}]$, as stated in Lemma \ref{l:combinatorial-weak}(2), and Lemma \ref{lemma2}. This remains true if we redefined $\Delta_{j + 1} := \delta = \Delta_{n}$, since the slopes $\{\beta_{j}\}$ are increasing (thus, $f$ is also $(\beta_{j + 1},0)$-superlinear on $[a_{j},a_{n}] = [a_{j},m]$);
\item $C$ is a Frostman $(\Delta,\gamma,O(\delta^{-\epsilon}))$-set thanks to the uniformity of $C$, and Lemma \ref{OSlemma}. Moreover, $\delta^{-\epsilon} \leq \Delta^{-\chi}$ by \eqref{form18}. Let $\bar{C} \subset C$ be a maximal $\Delta$-separated subset. Then $\bar{C}$ is a Frostman $(\Delta,\gamma,\Delta^{-\chi})$-set;
\item $\beta_{j + 1} + \gamma \geq \alpha' + \eta$ according to \eqref{form4}. 
\end{itemize}
These properties can be summarised by $(\alpha',\beta_{j + 1},\gamma) \in K$ (recall \eqref{def:K}). It then follows from the definition of $\chi$ (that is: Theorem \ref{thm:ABCConjecture} applied at scale $\Delta$) that there exists a set $C_{IJ} \subset \bar{C}$ of cardinality $|C_{IJ}| \geq (1 - \Delta^{\chi})|\bar{C}|$ such that 
\begin{equation}\label{form7} c \in C_{IJ} \quad \Longrightarrow \quad |\{a + cb : (a,b) \in \mathbf{G}\}|_{\Delta} \geq \Delta^{-\chi}|A_{I}|_{\Delta}, \end{equation} 
whenever $\mathbf{G} \subset A_{I} \times B_{J}$ satisfies $|\mathbf{G}| \geq \Delta^{\chi}|A_{I}||B_{J}|$. 

Next we define the point "$c$" whose existence is claimed in Proposition \ref{thm:ABC}. We start by noting that 
\begin{displaymath} \sum_{c \in \bar{C}} |\{I \times J \in \mathcal{D}_{\Delta_{j}}(A \times B) : c \in C_{IJ}\}| \geq (1 - \Delta^{\chi})|A|_{\Delta_{j}}|B|_{\Delta_{j}}|\bar{C}|. \end{displaymath}
Thus, there exists a point $c \in \bar{C} \subset C$ such that 
\begin{equation}\label{form8} |\{I \times J \in \mathcal{D}_{\Delta_{j}}(A \times B) : c \in C_{IJ}\}| \geq (1 - \Delta^{\chi/2})|A|_{\Delta_{j}}|B|_{\Delta_{j}}. \end{equation}
We claim that \eqref{form13} holds for this "$c$". Let $G \subset A \times B$ be arbitrary with $|G| \geq \delta^{\epsilon}|A||B|$. Then, thanks to the uniformity of $A \times B$,
\begin{displaymath} |\{I \times J \in \mathcal{D}_{\Delta_{j}}(A \times B) : |G \cap (I \times J)| \geq \delta^{2\epsilon}|(A \cap I) \times (B \cap J)|\}| \geq \delta^{2\epsilon}|A|_{\Delta_{j}}|B|_{\Delta_{j}}.\end{displaymath}
Combining this with \eqref{form8}, and noting that $\delta^{2\epsilon} \geq 2\Delta^{\chi/2}$ thanks to \eqref{form18}, we find
\begin{equation}\label{form47} |\{I \times J \in \mathcal{D}_{\Delta_{j}}(A \times B) : c \in C_{IJ} \text{ and } |G \cap (I \times J)| \geq \delta^{2\epsilon}|(A \cap I) \times (B \cap J)|\}| \geq \delta^{3\epsilon}|A|_{\Delta_{j}}|B|_{\Delta_{j}}. \end{equation}
Let $\mathcal{G} \subset \mathcal{D}_{\Delta_{j}}(A \times B)$ be the family of "good" squares $I \times J$ appearing in \eqref{form47}, thus satisfying $|G \cap (I \times J)| \geq \delta^{2\epsilon}|(A \cap I) \times (B \cap J)|$.

For $I \times J \in \mathcal{G}$ fixed, we define the \emph{heavy} $\Delta_{j + 1}$-squares
\begin{displaymath} \mathcal{H}_{IJ} := \{\mathfrak{i} \times \mathfrak{j} \in \mathcal{D}_{\Delta_{j + 1}}(\mathcal{G} \cap (I \times J)) : |G \cap (\mathfrak{i} \times \mathfrak{j})| \geq \delta^{3\epsilon}|(A \cap \mathfrak{i}) \times (B \cap \mathfrak{j})|\}. \end{displaymath}
It follows from $|G \cap (I \times J)| \geq \delta^{2\epsilon}|(A \cap I) \times (B \cap J)|$ and the uniformity of $A \times B$ that $|\mathcal{H}_{IJ}| \geq \delta^{3\epsilon}|A|_{\Delta_{j} \to \Delta_{j + 1}}|B|_{\Delta_{j} \to \Delta_{j + 1}} = \delta^{3\epsilon}|A_{I}||B_{J}|$. We set
\begin{displaymath} \mathbf{G}_{IJ} := (S_{I} \times S_{J})(\mathcal{H}_{IJ}) \subset A_{I} \times B_{J}. \end{displaymath} 
Thus, $|\mathbf{G}_{IJ}| = |\mathcal{H}_{IJ}| \geq \delta^{3\epsilon}|A_{I}||B_{J}| \geq \Delta^{\chi}|A_{I}||B_{J}|$, using also \eqref{form18}. Therefore, we may use \eqref{form7} applied to $\mathbf{G} = \mathbf{G}_{IJ}$  to deduce 
\begin{align} |\{a + cb : (a,b) \in \cup \mathcal{H}_{IJ}\}|_{\Delta_{j + 1}} & = |\{a + cb : (a,b) \in \cup \mathbf{G}_{IJ}\}|_{\Delta} \notag\\
&\label{form19} \stackrel{\eqref{form7}}{\geq} \Delta^{-\chi}|A_{I}|_{\Delta} = \Delta^{-\chi}|A|_{\Delta_{j} \to \Delta_{j + 1}}. \end{align}
To deduce a lower bound for $|\{a + cb: (a,b) \in G\}|_{\delta}$, we first note that for $\mathfrak{i} \times \mathfrak{j} \in \mathcal{H}_{IJ}$ fixed,
\begin{displaymath} |\{a + cb : (a,b) \in G \cap (\mathfrak{i} \times \mathfrak{j})\}|_{\delta} \geq \delta^{3\epsilon}|A \cap \mathfrak{i}| = \delta^{3\epsilon}|A|_{\Delta_{j + 1} \to \delta}, \end{displaymath} 
simply because there exists $b \in B \cap \mathfrak{j}$ such that $|\{a \in A \cap \mathfrak{i} : (a,b) \in G \cap (\mathfrak{i} \times \mathfrak{j})\}| \geq \delta^{3\epsilon}|A \cap \mathfrak{i}|$. As a consequence of this and \eqref{form19}, and recalling from \eqref{form18} that $\delta^{3\epsilon} \leq \Delta^{-\chi/2}$,
\begin{align*} |\{a + cb : (a,b) \in & \, G \cap (I \times J)\}|_{\delta}\\
& \gtrsim  |\{a + cb : (a,b) \in \cup \mathcal{H}_{IJ}\}|_{\Delta_{j + 1}} \cdot \delta^{3\epsilon}|A|_{\Delta_{j + 1} \to \delta} \geq \Delta^{-\chi/2}|A|_{\Delta_{j} \to \delta}. \end{align*}
This estimate is valid for all $I \times J \in \mathcal{G}$. Finally, we combine the information from the pieces $G \cap (I \times J)$. First, we fix a "good row": namely, recalling $|\mathcal{G}| \geq \delta^{3\epsilon}|A|_{\Delta_{j}}|B|_{\Delta_{j}}$ by \eqref{form47}, there exists $J \in \mathcal{D}_{\Delta_{j}}(B)$ (fixed from now on) such that $|\mathcal{G}_{J}| \geq \delta^{3\epsilon}|A|_{\Delta_{j}}$, where
\begin{displaymath} \mathcal{G}_{J} := \{I \in \mathcal{D}_{\Delta_{j}}(A) : I \times J \in \mathcal{G}\}. \end{displaymath}
Now, it is easy to check that the sets $\{a + cb : (a,b) \in (I \times J)\}$ have bounded overlap as $I$ varies in $\mathcal{D}_{\Delta_{j}}(\R)$ (and $c,J$ are fixed). Therefore, using also \eqref{form18},
\begin{displaymath} |\{a + cb : (a,b) \in G\}|_{\delta} \gtrsim \sum_{I \in \mathcal{G}_{J}} |\{a + cb : (a,b) \in G \cap (I \times J)\}|_{\delta} \gtrsim |\mathcal{G}_{J}| \cdot \Delta^{-\chi/2}|A|_{\Delta_{j} \to \delta} \geq \delta^{-\epsilon}|A|. \end{displaymath}
This completes the proof of Proposition \ref{thm:ABC}. \end{proof} 

\section{Proof of Theorem \ref{c3}}\label{s4}

In this section we prove Theorem \ref{c3}. We first extract from the main argument a reduction, which says that $C \subset [\tfrac{1}{2},1]$ without loss of generality. This would be straightforward if $C$ satisfied some "$2$-ends condition", but the set $C$ in Theorem \ref{c3} may be contained in a short interval $[0,\Delta]$ to begin with. The main idea of the proof is to observe that the condition Theorem \ref{c3}$(\Pi)$ is (somewhat) invariant with respect to rescaling $C$, see \eqref{form45}.

\begin{proposition}\label{prop2} It suffices to prove Theorem \ref{c3} under the additional hypothesis $C \subset [\tfrac{1}{2},1]$. \end{proposition} 

\begin{proof} Let $\alpha,\beta,\gamma,\eta$ be the constants from Theorem \ref{c3}. Let $\epsilon_{0} := \epsilon_{0}(\alpha,\beta,\gamma,\eta/4) > 0$ be the constant provided by the special case of Theorem \ref{c3}, where $C \subset [\tfrac{1}{2},1]$. We will apply the special case in the following form: if $A',B',C' \subset \rho \Z \cap [0,1]$ are sets satisfying the hypotheses of Theorem \ref{c3} at some (sufficiently small) scale $\rho > 0$ with $C' \subset [\tfrac{1}{2},1]$, and for each $c \in C'$ we are given a set $G_{c}' \subset A' \times B'$ with $|G_{c}'| \geq \rho^{\epsilon_{0}}|A'||B'|$, then 
\begin{equation}\label{form35} \sum_{c \in C'} |\{a + cb : (a,b) \in G_{c}'\}|_{\rho} \geq \rho^{-\epsilon_{0}}|A'||C'|. \end{equation} 
In other words, we have upgraded the existence of a single point $c \in C'$ to a statement that the "average" point $c \in C'$ satisfies $|\{a + cb : (a,b) \in G_{c}'\}|_{\rho} \geq \rho^{-\epsilon_{0}}|A'|$. The justification for this apparently stronger statement is the same as in Remark \ref{rem1} (the existence of a single good point $c \in C$ is formally equivalent to half of the points in $C$ being good).

Set
\begin{equation}\label{form36} \epsilon := \tfrac{1}{12}\eta \epsilon_{0}. \end{equation}
We claim that Theorem \ref{c3} holds (without the restriction $C \subset [\tfrac{1}{2},1]$) with this "$\epsilon$". Let $A,B,C \subset \delta \Z \cap [0,1]$ be as in the statement of Theorem \ref{c3}, in particular
\begin{equation}\label{form33} |B|^{\gamma}|C|^{\beta}\delta^{\beta\gamma} \geq \delta^{-\eta}. \end{equation}
Make a counter assumption: for each $c \in C$ there exists $G_{c} \subset A \times B$ with $|G_{c}| \geq \delta^{\epsilon}|A||B|$ such that
\begin{equation}\label{form31} |\{a + cb : (a,b) \in G_{c}\}|_{\delta} \leq \delta^{-\epsilon}|A|. \end{equation}

We claim that there exists a scale $\Delta = 2^{-k} \geq \tfrac{1}{2}\delta^{1 - \eta/2}$ (depending on $C$) such that
\begin{itemize}
\item[(a)] $|C \cap [0,\Delta]| \geq \Delta^{\eta/2}|C| \geq \delta^{\eta/2}|C|$, 
\item[(b)] $|C \cap (\tfrac{\Delta}{2},\Delta]| \geq (1 - 2^{-\eta/2})|C \cap [0,\Delta]|$.
\end{itemize}
Assume that condition (b) fails for $k = 0,\ldots,m$. An easy induction shows that
\begin{equation}\label{form30} |C \cap [0,2^{-k}]| \geq 2^{-(\eta/2)k}|C|, \qquad k \in \{0,\ldots,m + 1\}. \end{equation}
On the other hand, note that the hypothesis \eqref{form33} combined with $|B| \leq \delta^{-\beta}$ implies $|C| \geq \delta^{-\eta/\beta} \geq \delta^{-\eta}$. Therefore, since $C$ is $\delta$-separated,
\begin{displaymath} 2^{-m}/\delta \geq |C \cap [0,2^{-m}]| \geq 2^{-(\eta/2) m}|C| \geq 2^{-(\eta/2)m}\delta^{-\eta},\end{displaymath}
which can be rearranged to
\begin{displaymath} 2^{-m(1 - \eta/2)} \geq \delta^{1 - \eta} \quad \Longrightarrow \quad 2^{-m} \geq \delta^{1 - \eta/(2 - \eta)} \geq \delta^{1 - \eta/2}. \end{displaymath}
Therefore, if $m_{0} \in \N$ is the smallest integer satisfying $2^{-m_{0}} < \delta^{1 - \eta/2}$, (b) must hold for some $\Delta = 2^{-k}$, $k \in \{0,\ldots,m_{0}\}$. We let $m := k$ be the smallest such integer, and set $\Delta := 2^{-m}$. Then $\Delta \geq \tfrac{1}{2}\delta^{1 - \eta/2}$, and the minimality of $m$ combined with \eqref{form30} gives (a). 

Let $\Delta \geq \tfrac{1}{2}\delta^{1 - \eta/2}$ be a scale satisfying (a)-(b). Apply Corollary \ref{cor2} at scale $r := \delta/\Delta \leq \delta^{\eta}$ and with constant $2\epsilon$ to the set $B$. This produces disjoint Katz-Tao $(\delta/\Delta,\beta)$-sets $B_{1},\ldots,B_{N} \subset B$ satisfying
\begin{equation}\label{form32} |B_{j}| \gtrsim \delta^{2\epsilon}\Delta^{\beta}|B| \quad \text{and} \quad |B \, \setminus \, (B_{1} \cup \ldots \cup B_{N})| \leq \delta^{2\epsilon}|B|. \end{equation}
For $c \in C \cap [\tfrac{\Delta}{2},\Delta]$ fixed, note that $|G_{c}| \geq \delta^{\epsilon}|A||B|$ implies
\begin{displaymath} |G_{c} \cap (A \times (B_{1} \cup \ldots \cup B_{N}))| \geq  \tfrac{1}{2}\delta^{\epsilon}|A||B| \geq \delta^{\epsilon}|A|\sum_{j = 1}^{N}|B_{j}| \end{displaymath}
therefore
\begin{displaymath}  \sum_{j = 1}^{N}\sum_{c \in C \cap [\Delta/2,\Delta]} |G_{c} \cap (A \times B_{j})| \geq \tfrac{1}{2}\delta^{\epsilon}|A||C \cap [\tfrac{\Delta}{2},\Delta]| \sum_{j = 1}^{N} |B_{j}|. \end{displaymath}
This implies the existence of $j \in \{1,\ldots,N\}$ (fixed for the rest of the proof) such that
\begin{displaymath} \sum_{c \in C \cap [\Delta/2,\Delta]} |G_{c} \cap (A \times B_{j})| \geq \tfrac{1}{2}\delta^{\epsilon}|A||B_{j}||C \cap [\tfrac{\Delta}{2},\Delta]|. \end{displaymath}
Abbreviate $\bar{B} := B_{j}$ and $C' := C \cap [\tfrac{\Delta}{2},\Delta]$. Then, the previous inequality implies 
\begin{equation}\label{form37} \sum_{I \in \mathcal{D}_{\Delta}(A)}\sum_{c \in C'} |G_{c} \cap ((A \cap I) \times \bar{B})| \geq \tfrac{1}{2}\delta^{\epsilon}|A||\bar{B}||C'|. \end{equation}
For each $I \in \mathcal{D}_{\Delta}(A)$, write 
\begin{equation}\label{def:CI} C_{I} := \{c \in C' : |G_{c} \cap ((A \cap I) \times \bar{B})| \geq \tfrac{1}{8}\delta^{\epsilon}|A \cap I||\bar{B}|\}. \end{equation}
Then, set $\mathcal{G} := \{I \in \mathcal{D}_{\Delta}(A) : |C_{I}| \geq \tfrac{1}{8}\delta^{\epsilon}|C'|\}$. Then \eqref{form37} implies
\begin{displaymath} \sum_{I \in \mathcal{G}}\sum_{c \in C_{I}} |G_{c} \cap ((A \cap I) \times \bar{B})| \geq \tfrac{1}{4}\delta^{\epsilon}|A||\bar{B}||C'|, \end{displaymath}
and consequently
\begin{equation}\label{form34} \sum_{I \in \mathcal{G}} |A \cap I| \geq (|\bar{B}||C'|)^{-1} \sum_{I \in \mathcal{G}} \sum_{c \in C_{I}} |G_{c} \cap ((A \cap I) \times \bar{B})| \geq \tfrac{1}{4}\delta^{\epsilon}|A|. \end{equation}

For $I \in \mathcal{G}$ fixed, define $\bar{C}_{I} := \Delta^{-1}C_{I}$. Then $\bar{C}_{I} \subset [\tfrac{1}{2},1]$ is a Katz-Tao $(\delta/\Delta,\gamma)$-set with 
\begin{displaymath} |\bar{C}_{I}| \geq \tfrac{1}{8}\delta^{\epsilon} |C'| \gtrsim_{\eta} \delta^{\epsilon}|C \cap [0,\Delta]| \stackrel{\mathrm{(a)}}{\geq} \delta^{\epsilon + \eta/2}|C|, \end{displaymath}
and consequently (using also $3\epsilon \leq \eta/4$ by \eqref{form36}),
\begin{align} |\bar{B}|^{\gamma}|\bar{C}_{I}|^{\beta}\left(\frac{\delta}{\Delta} \right)^{\beta\gamma} & \stackrel{\eqref{form32}}{\gtrsim_{\eta}} (\delta^{2\epsilon}\Delta^{\beta}|B|)^{\gamma}(\delta^{\epsilon + \eta/2}|C|)^{\beta}\left(\frac{\delta}{\Delta} \right)^{\beta\gamma} \notag\\
&\label{form45} \stackrel{\eqref{form33}}{\geq} \delta^{-\eta/2 + 3\epsilon} \geq \left(\frac{\delta}{\Delta}\right)^{-\eta/4}.  \end{align} 
This means that $\bar{B},\bar{C}$ satisfy the hypotheses of the case of Theorem \ref{c3}, where $\bar{C} \subset [\tfrac{1}{2},1]$.

We proceed to define $(\delta/\Delta,\alpha)$-sets $\bar{A}_{I} \subset [0,1]$ to which the special case may be applied. For $I \in \mathcal{G}$, set $\bar{A}_{I} := S_{I}(A \cap I)$, where $S_{I}$ is the rescaling map taking $I$ to $[0,1)$. For $c \in C_{I}$, define also $\bar{G}_{c,I} := \{(S_{I}(a),b) : (a,b) \in G_{c} \cap ((A \cap I) \times \bar{B})\} \subset \bar{A}_{I} \times \bar{B}$. Then the definition of $c \in C_{I}$ (recall \eqref{def:CI}) implies
\begin{displaymath} |\bar{G}_{c,I}| \geq \tfrac{1}{8}\delta^{\epsilon}|\bar{A}_{I}||\bar{B}| \stackrel{\eqref{form36}}{\geq} \delta^{\eta \epsilon_{0}}|\bar{A}_{I}||\bar{B}| \geq (\delta/\Delta)^{\epsilon_{0}}|\bar{A}_{I}||\bar{B}|. \end{displaymath}
Consequently, by \eqref{form35} applied at scale $\rho := \delta/\Delta$ to the sets $\bar{A}_{I},\bar{B},C_{I},\bar{G}_{c,I}$,
\begin{align*} \sum_{c \in C_{I}} |\{a + cb : (a,b) \in G_{c} \cap ((A \cap I) \times \bar{B})\}|_{\delta} & = \sum_{c \in C_{I}} |\{a + (c/\Delta)b : (a,b) \in \bar{G}_{c,I}\}|_{\delta/\Delta}\\
& \geq \left(\tfrac{\delta}{\Delta} \right)^{-\epsilon_{0}}|\bar{A}_{I}||C_{I}| \geq \tfrac{1}{8}\delta^{\epsilon - \eta \epsilon_{0}}|A \cap I||C'|. \end{align*}
Finally, note that for $c \in C' \subset [0,\Delta]$ fixed, the sets $I + c[0,1]$ have bounded overlap as $I \in \mathcal{D}_{\Delta}(A)$ varies. Therefore,
\begin{align*} \sum_{c \in C'} |\{a + cb : (a,b) \in G_{c}\}|_{\delta} & \gtrsim \sum_{I \in \mathcal{G}} \sum_{c \in C_{I}}  |\{a + cb : (a,b) \in G_{c} \cap ((A \cap I) \times \bar{B})\}|_{\delta}\\
& \geq \tfrac{1}{8}\delta^{\epsilon - \eta\epsilon_{0}} |C'| \sum_{I \in \mathcal{G}} |A \cap I| \stackrel{\eqref{form34}}{\geq} \tfrac{1}{32} \delta^{2\epsilon - \eta \epsilon_{0}} |A||C'|. \end{align*}
Since $\eta \epsilon_{0} > 3\epsilon$ by \eqref{form36}, the estimate above yields $c \in C' \subset C$ such that $|\{a + cb : (a,b) \in G_{c}\}|_{\delta} > \delta^{-\epsilon}|A|$. This violates our counter assumption \eqref{form31}, and completes the proof. \end{proof} 

We then complete the proof of Theorem \ref{c3} in the case $C \subset [\tfrac{1}{2},1]$.

\begin{proof}[Proof of Theorem \ref{c3}] Thanks to Proposition \ref{prop2}, we may assume $C \subset [\tfrac{1}{2},1]$. We may also assume that $C$ is $\{2^{-jT}\}_{j = 1}^{m}$-uniform for any $T \geq 1$ large enough, using Lemma \ref{l:uniformization}. Next, we may assume that $A,B$ are $\{2^{-jT}\}_{j = 1}^{m}$-uniform. This reduction is slightly more complicated, but very similar to the one recorded at the start of the proof of Proposition \ref{thm:ABC}, so we leave the details to the reader.

We start by defining the constant "$\epsilon$" for which Theorem \ref{c3} holds. Fix $\alpha,\beta,\gamma,\eta$ as in the statement of Theorem \ref{c3}, and consider
\begin{displaymath} K := \{(\alpha,\beta',\gamma') \in \R^{3} : \beta' \in [\beta \eta/2,1] \text{ and } \gamma' \in [\max\{\eta/2,\alpha - \beta'  + \alpha \eta/2\},1]\}. \end{displaymath}
Then $K \subset \Omega_{\mathrm{ABC}}$ is compact, so Proposition \ref{thm:ABC} (and the second part of Remark \ref{rem1}) yields a constant $\epsilon_{0} := \epsilon_{0}(K) > 0$. Let $\tau = \tau(\eta/2) > 0$ be the constant provided by Lemma \ref{l:combinatorial-weak} applied with constant $\eta/2$. Fix an absolute constant $\mathbf{C} \geq 1$ to be determined later (at \eqref{form39}), and fix $\zeta > 0$ so small that 
\begin{equation}\label{def:epsilon} \xi := \max\Big\{\frac{\mathbf{C}}{\log_{2}(1/(4\zeta))},5\zeta\Big\} \leq \tfrac{1}{4}\epsilon_{0}\tau. \end{equation}
Finally, let $\epsilon > 0$ be smaller than the constant provided by the Balog-Szemer\'edi-Gowers theorem (Theorem \ref{t:BSG}) applied with constant $\zeta$. 

We now make a counter assumption: for every $c \in C$, there exists a subset $G_{c} \subset A \times B$ with $|G_{c}| \geq \delta^{\epsilon}|A||B|$ such that 
\begin{equation}\label{form25} |\{a + cb : (a,b) \in G_{c}\}|_{\delta} \leq \delta^{-\epsilon}|A|. \end{equation}

Let $f \colon [0,m] \to [0,m]$ be the branching function of $C$, and let $\{a_{j}\}_{j = 0}^{n}$ and $\{\gamma_{j}\}_{j = 1}^{n}$ be the sequences provided by Lemma \ref{l:combinatorial-weak} applied to $f$ and constant $\eta/2$. Since $C$ is a Katz-Tao $(\delta,\gamma)$-set, and the slopes $\gamma_{j}$ are increasing, Lemma \ref{l:combinatorial-weak}(2) and Remark \ref{rem6} imply
\begin{equation}\label{form21} \gamma_{j} \leq \gamma_{n} \leq s_{f}(a_{n - 1},m) \leq \gamma, \qquad j \in \{1,\ldots,n\}. \end{equation}
Write $\Delta_{j} := 2^{-a_{j}T}$, so $\delta = \Delta_{n} < \Delta_{n - 1} < \ldots < \Delta_{0} = 1$. By Lemma \ref{l:combinatorial-weak}(3),
\begin{equation}\label{form20} \prod_{j = 0}^{n - 1} \left(\frac{\Delta_{j + 1}}{\Delta_{j}} \right)^{-\gamma_{j + 1}} = 2^{T\sum_{j = 0}^{n - 1} (a_{j + 1} - a_{j})\gamma_{j + 1}} \geq 2^{T(f(m) - m \eta/2)} = |C|\delta^{\eta/2}. \end{equation} 
Moreover, it follows from the hypothesis $|B|^{\gamma}|C|^{\beta}\delta^{\beta\gamma} \geq \delta^{-\eta}$, and $|B| \leq \delta^{-\beta}$ (by the Katz-Tao $(\delta,\beta)$-property of $B$), that $|C| \geq \delta^{-\eta}$. Therefore \eqref{form20} implies 
\begin{equation}\label{form38} \gamma_{n} = \max \gamma_{j} \geq \eta - \eta/2 \geq \eta/2. \end{equation}
We finally recall from Lemma \ref{l:combinatorial-weak}(1) that $\Delta_{j + 1}/\Delta_{j} \leq \delta^{\tau}$.

We claim that there exists an index $j \in \{0,\ldots,n - 1\}$ such that
\begin{equation}\label{form5} \gamma_{j + 1} \geq \eta/2 \quad \text{and} \quad |B|_{\delta/\Delta_{j + 1} \to \delta/\Delta_{j}}^{\alpha/\beta}\left(\frac{\Delta_{j}}{\Delta_{j + 1}} \right)^{(\alpha/\gamma)\gamma_{j + 1}} \geq \left(\frac{\Delta_{j}}{\Delta_{j + 1}} \right)^{\alpha(1 + \eta/(2\gamma))}. \end{equation} 
To show this, let $n_{0} := \min\{0 \leq j \leq n - 1: \gamma_{j + 1} \geq \eta/2\}$. (Recall $\gamma_{n} \geq \eta/2$ by \eqref{form38}, thus the range over which we take the minimum is non-empty.) Now, assume \eqref{form5} fails for every index $j \in \{n_{0},\ldots,n - 1\}$. Taking products on both sides, and using $(\alpha/\gamma)\gamma_{j + 1} \leq \alpha \eta/(2\gamma)$ for $0 \leq j \leq n_{0} - 1$ (if $n_{0} = 0$, this range is empty, and the corresponding product $\Pi_{j = 0}^{n_{0} - 1}$ below should be interpreted to have value $1$),
\begin{align*} |B|^{\alpha/\beta}|C|^{\alpha/\gamma}\delta^{\alpha \eta/(2\gamma)} & \stackrel{\eqref{form20}}{\leq} \prod_{j = 0}^{n - 1} |B|^{\alpha/\beta}_{\delta/\Delta_{j + 1} \to \delta/\Delta_{j}}\left(\frac{\Delta_{j}}{\Delta_{j + 1}} \right)^{(\alpha/\gamma)\gamma_{j + 1}}\\
& \stackrel{(\ast)}{<} \prod_{j = 0}^{n_{0} - 1} |B|_{\delta/\Delta_{j + 1} \to \delta/\Delta_{j}}^{\alpha/\beta}\left(\frac{\Delta_{j}}{\Delta_{j + 1}} \right)^{\alpha \eta/(2\gamma)} \prod_{j = n_{0}}^{n - 1} \left(\frac{\Delta_{j}}{\Delta_{j + 1}} \right)^{\alpha(1 + \eta/(2\gamma))}\\
& = \delta^{-\alpha\eta/(2\gamma)}|B|_{\delta/\Delta_{n_{0}} \to \delta}^{\alpha/\beta}\left(\frac{\delta}{\Delta_{n_{0}}} \right)^{-\alpha} \leq \delta^{-\alpha(1 + \eta/(2\gamma))},  \end{align*} 
where the $(\delta,\beta)$-Katz-Tao property of $B$ was used on the last line. The strict inequality $(\ast)$ arises from the fact that $\gamma_{n} \geq \eta/2$, so if the strict converse inequality in \eqref{form5} holds for all $j \in \{n_{0},\ldots,n - 1\}$, then in particular it holds at least once (with index $j = n - 1$). The previous aligned inequality can be rearranged to $|B|^{\gamma}|C|^{\beta}\delta^{\beta \gamma} < \delta^{-\beta \eta} \leq \delta^{-\eta}$, so a contradiction has been reached.

Let $j \in \{0,\ldots,n - 1\}$ be an index such that \eqref{form5} holds; this index is fixed for the remainder of the proof. Now we return to our counter assumption \eqref{form25}. Fix an interval $L \in \mathcal{D}_{\Delta_{j}}(C)$ arbitrarily. Instead of \eqref{form25}, we would prefer to know that 
\begin{displaymath} |\{a + (c - c_{0})b : (a,b) \in G_{c}\}| \leq \delta^{-\epsilon}|A|, \qquad c \in C \cap L, \end{displaymath}
where $c_{0} \in L$ is fixed. This can be achieved at the cost of minor refinements, and replacing $\epsilon$ by the constant $\xi > 0$ from \eqref{def:epsilon}.

\begin{claim}\label{c1} Let $\xi > 0$ be the constant defined in \eqref{def:epsilon}. The following objects exist:
\begin{itemize}
\item A $\Delta_{j + 1}$-separated subset $\bar{C} \subset C \cap L$ with $|\bar{C}| \geq \delta^{\xi}|C|_{\Delta_{j} \to \Delta_{j + 1}}$;
\item A point $c_{0} \in C \cap L$;
\item For each $c \in \bar{C}$ a set $\bar{G}_{c} \subset A \times B$ with $|\bar{G}_{c}| \geq \delta^{\xi}|A||B|$ such that
\end{itemize}
\begin{equation}\label{form28} |\{a + (c - c_{0})b : (a,b) \in \bar{G}_{c}\}|_{\delta} \leq \delta^{-\xi}|A|. \end{equation}

  \end{claim} 

\begin{proof} For each $c \in C \cap L \subset [\tfrac{1}{2},1]$, apply the Balog-Szemer\'edi-Gowers theorem (Theorem \ref{t:BSG}) to extract subsets $A_{c} \subset A$ and $B_{c} \subset B$ with the properties
\begin{equation}\label{form27} |A_{c}||B_{c}| \geq \delta^{\zeta}|A||B| \quad \text{and} \quad |A_{c} + cB_{c}|_{\delta} \leq \delta^{-\zeta}|A|. \end{equation}
By Cauchy-Schwarz,
\begin{displaymath} \sum_{c,c' \in C \cap L} |(A_{c} \times B_{c}) \cap (A_{c'} \times B_{c'})| \geq \delta^{2\zeta}|C \cap L|^{2}|A||B|. \end{displaymath}
Consequently, there exists $c_{0} \in C \cap L$, and $C_{0} \subset C \cap L$ with $|C_{0}| \geq \delta^{3\zeta}|C \cap L|$ such that
\begin{equation}\label{form26} |(A_{c_{0}} \times B_{c_{0}}) \cap (A_{c} \times B_{c})| \geq \delta^{3\zeta}|A||B|, \qquad c \in C_{0}. \end{equation}
Write $\bar{B}_{c} := B_{c_{0}} \cap B_{c}$, so $|\bar{B}_{c}| \geq \delta^{3\zeta}|B|$ for $c \in C_{0}$. By \eqref{form27} and \eqref{form26},
\begin{displaymath} |(A_{c_{0}} \cap A_{c}) + c_{0}\bar{B}_{c}|_{\delta} \leq \delta^{-4\zeta}|A_{c_{0}} \cap A_{c}|, \qquad c \in C_{0}. \end{displaymath}
Therefore, by Lemma \ref{TVLemma}, for any $N \in \N$ fixed, there exists $A_{c}' \subset A_{c_{0}} \cap A_{c}$ with $|A_{c}'| \sim |A_{c_{0}} \cap A_{c}|$ such that
\begin{displaymath} |A_{c}' - c_{0}\bar{B}_{c}|_{\delta} \lesssim_{N} (\delta^{-4\zeta})^{2^{N}}|A_{c_{0}} \cap A_{c}|^{1 + 1/N}. \end{displaymath}
Applying this with $N := \tfrac{1}{2} \log_{2}(1/(4\zeta))$ (or the integer part thereof), and recalling that
\begin{displaymath} \xi = \max\{\mathbf{C}/\log_{2}(1/(4\zeta)),5\zeta\} \end{displaymath}
(for a suitable absolute constant $\mathbf{C} \geq 1$), we find
\begin{equation}\label{form39} |A_{c}' - c_{0}\bar{B}_{c} |_{\delta} \lesssim_{\zeta} \delta^{-2\sqrt{\zeta}}|A_{c_{0}} \cap A_{c}|^{1 + 2/\log_{2}(1/(4\zeta))} \leq \delta^{-\xi/3}|A_{c}'|, \end{equation}
using the crude bounds $|A_{c_{0}} \cap A_{c}| \leq \delta^{-1}$, and $2\sqrt{\zeta} \lesssim 1/\log_{2}(1/(4\zeta))$.

Since $A_{c}' \subset A_{c}$, and $\delta > 0$ is small enough, we then have both
\begin{displaymath} |A_{c}' + c\bar{B}_{c} |_{\delta} \stackrel{\eqref{form27}}{\leq} \delta^{-\xi/2}|A_{c}'| \quad \text{and} \quad |A_{c}' - c_{0}\bar{B}_{c} |_{\delta} \stackrel{\eqref{form39}}{\leq} \delta^{-\xi/2}|A_{c}'|, \end{displaymath}
Therefore, by the generalised Pl\"unnecke-Ruzsa inequality, Lemma \ref{PRIneq}, there exists yet another subset $\bar{A}_{c} \subset A_{c}'$ of size $|\bar{A}_{c}| \sim |A_{c}'|$ such that, writing $\bar{G}_{c} := \bar{A}_{c} \times \bar{B}_{c}$, 
\begin{displaymath} |\{a + (c - c_{0})b : (a,b) \in \bar{G}_{c}\}|_{\delta} = |\bar{A}_{c} + (c - c_{0})\bar{B}_{c}|_{\delta} \lesssim \delta^{-\xi}|A|, \quad c \in C_{0} \subset C \cap L. \end{displaymath}
Note that $|\bar{G}_{c}| \geq \delta^{\xi}|A||B|$ by \eqref{form26}. To complete the proof of the claim, let $\bar{C} \subset C_{0}$ be a maximal $\Delta_{j + 1}$-separated subset. Thanks to the uniformity of $C$, and since $|C_{0}| \geq \delta^{3\zeta}|C \cap L|$, it holds $|\bar{C}| \geq \delta^{\xi}|C|_{\Delta_{j} \to \Delta_{j + 1}}$. \end{proof}

To make use of Claim \ref{c1}, fix
\begin{displaymath} I = [x_{I},x_{I} + \delta \Delta_{j}/\Delta_{j + 1}) \in \mathcal{D}_{\delta \Delta_{j}/\Delta_{j + 1}}(A) \quad \text{and} \quad J = [y_{J},y_{J} + \delta/\Delta_{j + 1}) \in \mathcal{D}_{\delta/\Delta_{j + 1}}(B). \end{displaymath}
Write
\begin{displaymath} \begin{cases} A_{I} := S_{I}(\mathcal{D}_{\delta}(A \cap I)) \subset \mathcal{D}_{\Delta_{j + 1}/\Delta_{j}}([0,1)), \\ B_{J} := S_{J}(\mathcal{D}_{\delta/\Delta_{j}}(B \cap J)) \subset \mathcal{D}_{\Delta_{j + 1}/\Delta_{j}}([0,1)), \\ C' := \Delta_{j}^{-1}(\bar{C} - c_{0}), \end{cases} \end{displaymath}
where 
\begin{equation}\label{form11} S_{I}(x) = (x - x_{I}) \cdot \Delta_{j + 1}/(\delta \Delta_{j}) \quad \text{and} \quad S_{J}(y) = (y - y_{J}) \cdot \Delta_{j + 1}/\delta \end{equation}
are the rescaling maps sending $I,J$ to $[0,1)$.

Then the following properties hold for $\Delta := \Delta_{j + 1}/\Delta_{j}$:
\begin{itemize}
\item $A_{I}$ is a Katz-Tao $(\Delta,\alpha)$-set with $|A_{I}| = |A|_{\delta\Delta_{j}/\Delta_{j + 1} \to \delta}$.
\item $|B_{J}| = |B|_{\delta/\Delta_{j + 1} \to \delta/\Delta_{j}} =: \Delta^{-\beta'}$ with $\beta' \geq \beta\eta/2$. The lower bound for $\beta'$ follows from \eqref{form5} and $\gamma_{j + 1} \leq \gamma$, as recorded in \eqref{form21}.
\item $C' \subset [-1,1]$ is a $\Delta$-separated $(\Delta,\gamma',O_{T}(\delta^{-\xi}))$-Frostman set with 
\begin{displaymath} \gamma' := \gamma_{j + 1} \stackrel{\eqref{form5}}{\geq} \eta/2. \end{displaymath}
This is true, since $\bar{C} \subset C \cap L$ satisfies $|\bar{C}| \geq \delta^{\xi}|C|_{\Delta_{j} \to \Delta_{j + 1}}$, and the $\Delta_{j}^{-1}$-renormalisation of $C \cap L$ is $(\Delta,\gamma',O_{T}(1))$-Frostman by Lemma \ref{lemma2}. Note also that $\delta^{-\xi} \leq \Delta^{-\epsilon_{0}}$ by the choice of $\xi$ at \eqref{def:epsilon}, and the choice of $\epsilon_{0}$ at \eqref{def:epsilon}.

\item As a consequence of \eqref{form5}, and $\alpha \leq \min\{\beta,\gamma\}$,
\begin{equation}\label{form42} \Delta^{-\beta' - \gamma'} \geq |B|^{\alpha/\beta}_{\delta/\Delta_{j + 1} \to \delta/\Delta_{j}}\Delta^{-(\alpha/\gamma)\gamma_{j + 1}} \geq \Delta^{-\alpha(1 + \eta/(2\gamma))} \quad \Longrightarrow \quad \beta' + \gamma' \geq \alpha + \tfrac{\alpha \eta}{2}. \end{equation}
\end{itemize}
Therefore $(\alpha,\beta',\gamma') \in K$, and Proposition \ref{thm:ABC} (via the choice of "$\epsilon_{0}$") implies that there exists a subset $C_{IJ} \subset \bar{C}$ with $|C_{IJ}| \geq (1 - \Delta^{\epsilon_{0}})|\bar{C}|$ such that the following holds for all $c \in C_{IJ}$: if $\mathbf{G} \subset \cup (A_{I} \times B_{J})$ is a $\Delta$-separated set with $|\mathbf{G}| \geq \Delta^{\epsilon_{0}}|A_{I}||B_{J}|$, then
\begin{equation}\label{form6} |\{a + ((c - c_{0})/\Delta_{j})b : (a,b) \in \mathbf{G}\}|_{\Delta} \geq \Delta^{-\epsilon_{0}}|A_{I}| = \Delta^{-\epsilon_{0}}|A|_{\delta \Delta_{j}/\Delta_{j + 1} \to \delta}. \end{equation} 
From this point on, the argument is quite similar to that in Proposition \ref{thm:ABC}. Abbreviate
\begin{displaymath} \mathcal{D}(A \times B) := \mathcal{D}_{\delta \Delta_{j}/\Delta_{j + 1}}(A) \times \mathcal{D}_{\delta/\Delta_{j + 1}}(B). \end{displaymath}
By double counting, there exists a point $c \in \bar{C}$ such that
\begin{equation}\label{form9} |\{I \times J \in  \mathcal{D}(A \times B) : c \in C_{IJ}\}| \geq (1 - \Delta^{\epsilon_{0}/2})|A|_{\delta \Delta_{j}/\Delta_{j + 1}}|B|_{\delta/\Delta_{j + 1}}. \end{equation}
Fix this $c \in \bar{C} \subset C$ for the remainder of the proof, and abbreviate $G := \bar{G}_{c}$ (the set in Claim \ref{c1}). By $|G| \geq \delta^{\xi}|A||B|$ and the uniformity of $A \times B$,
\begin{displaymath} |\{I \times J \in  \mathcal{D}(A \times B) : |G \cap (I \times J)| \geq \delta^{2\xi}|(A \cap I) \times (B \cap J)|\}| \geq \delta^{2\xi}|A|_{\delta \Delta_{j}/\Delta_{j + 1}}|B|_{\delta/\Delta_{j + 1}}. \end{displaymath}
Combining this with \eqref{form9}, and noting $\delta^{2\xi} \geq 2\Delta^{\epsilon_{0}}$ by \eqref{def:epsilon},
\begin{displaymath} |\{I \times J : c \in C_{IJ} \text{ and } |G \cap (I \times J)| \geq \delta^{2\xi}|(A \cap I) \times (B \cap J)|\}| \geq \delta^{3\xi}|A|_{\delta \Delta_{j}/\Delta_{j + 1}}|B|_{\delta/\Delta_{j + 1}}. \end{displaymath}
The rectangles $I \times J \in \mathcal{D}(A \times B)$, as above, are called \emph{good} and denoted $\mathcal{G}$. Fix $I \times J \in \mathcal{G}$. To apply \eqref{form6}, we need to produce a set $\mathbf{G}_{IJ} \subset \cup (A_{I} \times B_{J})$. We do so by setting first 
\begin{equation}\label{form22} \mathcal{H}_{IJ} := \{\mathfrak{i} \times \mathfrak{j} \in \mathcal{D}_{\delta}(A \cap I) \times \mathcal{D}_{\delta/\Delta_{j}}(B \cap J) : G \cap (\mathfrak{i} \times \mathfrak{j}) \neq \emptyset\}. \end{equation}
Define also $G_{IJ} \subset G \cap (I \times J)$ by selecting a single point of $G$ from each $\mathfrak{i} \times \mathfrak{j} \in \mathcal{H}_{IJ}$. Then, put $\mathbf{G}_{IJ} := \{(S_{I}(a),S_{J}(b)) : (a,b) \in H_{IJ}\} \subset \cup (A_{I} \times B_{I})$. With this notation,
\begin{displaymath} \delta^{2\xi}|A|_{\delta \Delta_{j}/\Delta_{j + 1} \to \delta}|B|_{\delta/\Delta_{j + 1} \to \delta} = \delta^{2\xi}|(A \cap I) \times (B \cap J)| \leq |G \cap (I \times J)| \leq |G_{IJ}||B|_{\delta/\Delta_{j} \to \delta}, \end{displaymath}
which implies 
\begin{displaymath} |\mathbf{G}_{IJ}| = |G_{IJ}| \geq \delta^{2\xi}|A|_{\delta \Delta_{j}/\Delta_{j + 1} \to \delta}|B|_{\delta/\Delta_{j + 1} \to \delta/\Delta_{j}} = \delta^{2\xi}|A_{I}||B_{J}| \stackrel{\eqref{def:epsilon}}{\geq} \Delta^{\epsilon_{0}}|A_{I}||B_{J}|. \end{displaymath}
The set $\mathbf{G}_{IJ}$ is $\Delta$-separated, because $S_{I} \times S_{J}$ sends the rectangles $\mathfrak{i} \times \mathfrak{j}$ to $\Delta$-squares. We may now deduce from \eqref{form6} that
\begin{equation}\label{form10} |\{a + ((c - c_{0})/\Delta_{j})b : (a,b) \in \mathbf{G}_{IJ}\}|_{\Delta} \geq \Delta^{-\epsilon_{0}}|A|_{\delta\Delta_{j}/\Delta_{j + 1} \to \delta}, \qquad I \times J \in \mathcal{G}.\end{equation}
Finally, we need to relate the size of the sets $\{a + ((c - c_{0})/\Delta_{j})^{-1}b : (a,b) \in \mathbf{G}_{IJ}\}$ to the size of $\{a + (c - c_{0})b : (a,b) \in G\}$. To this end, we first record the relation (recall \eqref{form11})
\begin{displaymath} S_{I}(a) + (\Delta_{j}^{-1}(c - c_{0}))S_{J}(b) = (a + (c - c_{0})b) \cdot \frac{\Delta}{\delta} - (x_{I} + (c - c_{0})y_{J}) \cdot \frac{\Delta}{\delta} \end{displaymath} 
for $a \in I$ and $b \in J$. This implies
\begin{displaymath} \{a + ((c - c_{0})/\Delta_{j})b : (a,b) \in \mathbf{G}_{IJ}\} = \frac{\Delta}{\delta} \cdot \{a + (c - c_{0})b : (a,b) \in H_{IJ}\} - w_{IJ}, \end{displaymath}
with $w_{IJ}$ equal to the constant $w_{IJ} = (x_{I} + (c - c_{0})y_{J}) \cdot \frac{\Delta}{\delta}$, and in particular
\begin{align} |\{a + (c - c_{0})b : (a,b) \in G \cap (I \times J)\}|_{\delta} & \geq |\{a + (c - c_{0})b : (a,b) \in G_{IJ}\}|_{\delta} \notag\\
&\label{form12} = |\{a + ((c - c_{0})/\Delta_{j})b : (a,b) \in \mathbf{G}_{IJ}\}|_{\Delta}. \end{align} 
We also observe that for $c \in L$ and $J \in \mathcal{D}_{\delta/\Delta_{j + 1}}(B)$ fixed, the intervals $I + (c - c_{0})J$ (which contain $\{a + (c - c_{0})b : (a,b) \in G \cap (I \times J)\}$) have bounded overlap as $I \in \mathcal{D}_{\delta \Delta_{j}/\Delta_{j + 1}}(\R)$ varies; indeed 
\begin{displaymath} I + (c - c_{0})J \subset B(x_{I} + (c - c_{0})y_{J},2\delta \Delta_{j}/\Delta_{j + 1}), \end{displaymath} 
and the (left end-)points $x_{I}$ are $(\delta \Delta_{j}/\Delta_{j + 1})$-separated. This is where we needed that $|c - c_{0}| \leq \Delta_{j}$, explaining the purpose of proving Claim \ref{c1}.

Finally, recall from above \eqref{form22} that $|\mathcal{G}| \geq \delta^{3\xi}|A|_{\delta\Delta_{j}/\Delta_{j + 1}}|B|_{\delta/\Delta_{j + 1}}$. Consequently, we may fix $J \in \mathcal{D}_{\delta/\Delta_{j + 1}}(B)$ such that 
\begin{equation}\label{form23} |\mathcal{G}_{J}| \geq \delta^{3\xi}|A|_{\delta\Delta_{j}/\Delta_{j + 1}}, \end{equation}
where $\mathcal{G}_{J} := \{I \in \mathcal{D}_{\delta \Delta_{j}/\Delta_{j + 1}}(A) : I \times J \in \mathcal{G}\}$.

Using the bounded overlap of the intervals $I + (c - c_{0})J$ for this "$J$", we finally find
\begin{align*} |\{a + (c - c_{0})b : (a,b) \in G\}|_{\delta} & \gtrsim \sum_{I \in \mathcal{G}_{J}} |\{a + (c - c_{0})b : (a,b) \in G \cap (I \times J)\}|_{\delta}\\
& \stackrel{\eqref{form12}}{\geq} \sum_{I \in \mathcal{G}_{J}} |\{a + (c - c_{0})b : (a,b) \in \mathbf{G}_{IJ}\}|_{\Delta}\\
& \stackrel{\eqref{form10}}{\geq} \Delta^{-\epsilon_{0}}|\mathcal{G}_{J}| |A \cap I|_{\delta \Delta_{j}/\Delta_{j + 1} \to \delta}\\
& \geq \Delta^{-\epsilon_{0}}\delta^{3\xi} |A| \stackrel{\eqref{def:epsilon}}{\geq} \delta^{-2\xi}|A|. \end{align*} 
Recalling that $G = \bar{G}_{c}$, the estimate above contradicts \eqref{form28} and completes the proof. \end{proof} 

\section{Proof of Theorem \ref{c4}}\label{s5}

 In this section we indicate the changes -- or really simplifications -- to the previous argument needed to prove Theorem \ref{c4}, where the set $C$ is $(\delta,\eta,\delta^{-\epsilon})$-Frostman. The main point is that we do not have to expend effort in (the counterpart of) \eqref{form5} to establish $\gamma_{j + 1} > 0$. This is automatic by the Frostman property of $C$. Since ensuring $\gamma_{j + 1} > 0$ was the only place in the proof of Theorem \ref{c3} where the Katz-Tao $(\delta,\beta)$-set property of $B$ was used, the hypothesis can be omitted from Theorem \ref{c4}.

\begin{proof}[Proof of Theorem \ref{c4}] We start by defining "$\epsilon$" for which Theorem \ref{c4} holds. Let
\begin{displaymath} K := \{(\alpha',\beta',\gamma') \in \R^{3} : \alpha' = \alpha, \, \beta' \in [\alpha \eta/2,1] \text{ and } \gamma' \in [\max\{\eta/2,\alpha' - \beta'  + \alpha \eta/2\},1]\}. \end{displaymath}
Then $K \subset \Omega_{\mathrm{ABC}}$ is compact, so Proposition \ref{thm:ABC} yields a constant $\epsilon_{0} := \epsilon_{0}(K) > 0$. Fix a sufficiently large absolute constant $\mathbf{C} \geq 1$, and fix $\zeta > 0$ so small that 
\begin{equation}\label{def:epsilonA} \xi := \max\Big\{\frac{\mathbf{C}}{\log_{2}(1/(4\zeta))},5\zeta\Big\} \leq \tfrac{1}{4}\epsilon_{0}\tau, \end{equation}
where $\tau = \tau(\eta/2) > 0$ is the constant given by Lemma \ref{l:combinatorial-weak} applied with $\eta/2$. Then, let 
\begin{displaymath} \epsilon := \tfrac{1}{2} \min\{\epsilon_{\mathrm{BSG}}(\zeta),\tau/4\} \cdot \eta > 0, \end{displaymath}
where $\epsilon_{\mathrm{BSG}}(\zeta) > 0$ is given by the Balog-Szemer\'edi-Gowers theorem (Theorem \ref{t:BSG}) applied with constant $\zeta$.

We make a counter assumption: for every $c \in C$, there exists a subset $G_{c} \subset A \times B$ with $|G_{c}| \geq \delta^{\epsilon}|A||B|$ such that 
\begin{equation}\label{form25a} |\{a + cb : (a,b) \in G_{c}\}|_{\delta} \leq \delta^{-\epsilon}|A|. \end{equation}
Since $C$ is $(\delta,\eta,\delta^{-\epsilon})$-Frostman, it holds that $|C \cap [0,\delta^{2\epsilon/\eta}]| \leq \delta^{\epsilon}|C|$. This allows us to assume with no loss of generality that 
\begin{equation}\label{form43} C \subset [\delta^{2\epsilon/\eta},1] \subset [\delta^{\epsilon_{\mathrm{BSG}}(\zeta)},1]. \end{equation}
As before, we may assume that $A,B,C$ are all $\{2^{-jT}\}_{j = 1}^{m}$-uniform for some $T \geq 1$. Let $f \colon [0,m] \to [0,m]$ be the branching function of $C$, and let $\{a_{j}\}_{j = 0}^{n}$ and $\{\gamma_{j}\}_{j = 1}^{n}$ be the sequences provided by Lemma \ref{l:combinatorial-weak} applied to $f$ and constant $\eta/2$. Write $\Delta_{j} := 2^{-a_{j}T}$. 
Since $C$ is a Katz-Tao $(\delta,\gamma)$-set, and the slopes $\gamma_{j}$ are increasing, Lemma \ref{l:combinatorial-weak}(2) and Remark \ref{rem6} imply
\begin{equation}\label{form21a} \eta/2 \leq \gamma_{1} \leq \gamma_{j} \leq \gamma_{n} \leq s_{f}(a_{n - 1},m) \leq \gamma, \qquad j \in \{1,\ldots,n\}. \end{equation}
The lower bound $\gamma_{1} \geq \eta/2$ follows from the $(\delta,\eta,\delta^{-\epsilon})$-Frostman hypothesis, which (together with the uniformity of $C$) implies that $C$ is $(\Delta_{1},\eta,\delta^{-\epsilon})$-Frostman, and therefore $(\Delta_{1},\eta,\Delta_{1}^{-\eta/2})$-Frostman by \eqref{form21a} (since $\Delta_{1} \leq \delta^{\tau}$ by Lemma \ref{l:combinatorial-weak}(1)).

By Lemma \ref{l:combinatorial-weak}(3),
\begin{equation}\label{form20a} \prod_{j = 0}^{n - 1} \left(\frac{\Delta_{j + 1}}{\Delta_{j}} \right)^{-\gamma_{j + 1}} = 2^{T\sum_{j = 0}^{n - 1} (a_{j + 1} - a_{j})\gamma_{j + 1}} \geq 2^{T(f(m) - m \eta/2)} = |C|\delta^{\eta/2}. \end{equation} 

We claim that there exists an index $j \in \{0,\ldots,N - 1\}$ such that
\begin{equation}\label{form5a} \quad |B|_{\delta/\Delta_{j + 1} \to \delta/\Delta_{j}}\left(\frac{\Delta_{j}}{\Delta_{j + 1}} \right)^{(\alpha/\gamma)\gamma_{j + 1}} \geq \left(\frac{\Delta_{j}}{\Delta_{j + 1}} \right)^{\alpha(1 + \eta/(2\gamma))}. \end{equation} 
Indeed, if this fails for every $j \in \{0,\ldots,N - 1\}$, then 
\begin{align*} |B||C|^{\alpha/\gamma}\delta^{\alpha \eta/(2\gamma)} & \stackrel{\eqref{form20a}}{\leq} \prod_{j = 0}^{n - 1} |B|_{\delta/\Delta_{j + 1} \to \delta/\Delta_{j}}\left(\frac{\Delta_{j}}{\Delta_{j + 1}} \right)^{(\alpha/\gamma)\gamma_{j + 1}}\\
& < \prod_{j = 0}^{n - 1} \left(\frac{\Delta_{j}}{\Delta_{j + 1}} \right)^{\alpha(1 + \eta/(2\gamma))} = \delta^{-\alpha - \alpha \eta/(2\gamma)}. \end{align*}
This can be rearranged to $|B|^{\gamma}|C|^{\alpha}\delta^{\alpha \gamma} < \delta^{-\alpha \eta} \leq \delta^{-\eta}$, so a contradiction ensues.

Let $j \in \{0,\ldots,N - 1\}$ be an index such that \eqref{form5a} holds; this index is fixed for the remainder of the proof. Now we return to our counter assumption \eqref{form25a}. Fix an interval $L \in \mathcal{D}_{\Delta_{j}}(C)$ arbitrarily. The proof of the next claim is the same as the proof of Claim \ref{c1}, so the argument is omitted. Here we need that $C \cap L \subset [\delta^{\epsilon_{\mathrm{BSG}}(\zeta)},1]$, so the application of the Balog-Szemer\'edi-Gowers lemma is legitimate.
\begin{claim} Let $\xi > 0$ be the constant defined at \eqref{def:epsilonA}. The following objects exist:
\begin{itemize}
\item A $\Delta_{j + 1}$-separated subset $\bar{C} \subset C \cap L$ with $|\bar{C}| \geq \delta^{\xi}|C|_{\Delta_{j} \to \Delta_{j + 1}}$;
\item A point $c_{0} \in C \cap L$;
\item For each $c \in \bar{C}$ a set $\bar{G}_{c} \subset A \times B$ with $|\bar{G}_{c}| \geq \delta^{\xi}|A||B|$ such that
\end{itemize}
\begin{equation}\label{form28a} |\{a + (c - c_{0})b : (a,b) \in \bar{G}_{c}\}|_{\delta} \leq \delta^{-\xi}|A|. \end{equation}
\end{claim}

To make use of \eqref{form28a}, fix
\begin{displaymath} I = [x_{I},x_{I} + \delta \Delta_{j}/\Delta_{j + 1}) \in \mathcal{D}_{\delta \Delta_{j}/\Delta_{j + 1}}(A) \quad \text{and} \quad J = [y_{J},y_{J} + \delta/\Delta_{j + 1}) \in \mathcal{D}_{\delta/\Delta_{j + 1}}(B). \end{displaymath}
Write
\begin{displaymath} \begin{cases} A_{I} := S_{I}(\mathcal{D}_{\delta}(A \cap I)) \subset \mathcal{D}_{\Delta_{j + 1}/\Delta_{j}}([0,1)), \\ B_{J} := S_{J}(\mathcal{D}_{\delta/\Delta_{j}}(B \cap J)) \subset \mathcal{D}_{\Delta_{j + 1}/\Delta_{j}}([0,1)), \\ C' := \Delta_{j}^{-1}(\bar{C} - c_{0}), \end{cases} \end{displaymath}
where $S_{I}$ and $S_{J}$ are the rescaling maps as in \eqref{form11}. Then the following properties hold for $\Delta := \Delta_{j + 1}/\Delta_{j}$:
\begin{itemize}
\item $A_{I}$ is a Katz-Tao $(\Delta,\alpha)$-set with $|A_{I}| = |A|_{\delta\Delta_{j}/\Delta_{j + 1} \to \delta}$.
\item $|B_{J}| = |B|_{\delta/\Delta_{j + 1} \to \delta/\Delta_{j}} =: \Delta^{-\beta'}$ with $\beta' \geq \alpha\eta/2$. The lower bound for $\beta'$ follows from \eqref{form5a} and $\gamma_{j + 1} \leq \gamma$, as recorded in \eqref{form21a}.
\item $C' \subset [-1,1]$ is a $\Delta$-separated $(\Delta,\gamma',O_{T}(\delta^{-\xi}))$-Frostman set with
\begin{displaymath} \gamma' := \gamma_{j + 1} \stackrel{\eqref{form21a}}{\geq} \eta/2. \end{displaymath}
This is true, since $\bar{C} \subset C \cap L$ satisfies $|\bar{C}| \geq \delta^{\xi}|C|_{\Delta_{j} \to \Delta_{j + 1}}$, and the $\Delta_{j}^{-1}$-renormalisation of $C \cap L$ is $(\Delta,\gamma',O_{T}(1))$-Frostman by Lemma \ref{lemma2}. Note also that $\delta^{-\xi} \leq \Delta^{-\epsilon_{0}}$ by the choice of $\xi$ at \eqref{def:epsilonA}, and the choice of $\epsilon_{0}$ at \eqref{def:epsilonA}.

\item As a consequence of \eqref{form5a}, and $\alpha \leq \gamma$,
\begin{displaymath} \Delta^{-\beta' - \gamma'} \geq |B|_{\delta/\Delta_{j + 1} \to \delta/\Delta_{j}}\Delta^{-(\alpha/\gamma)\gamma_{j + 1}} \geq \Delta^{-\alpha(1 + \eta/(2\gamma))} \quad \Longrightarrow \quad \beta' + \gamma' \geq \alpha + \tfrac{\alpha \eta}{2}. \end{displaymath}
\end{itemize}
From this point on, one may follow \emph{verbatim} the proof of Theorem \ref{c3} below \eqref{form42}. The conclusion is the existence of a point $c \in \bar{C} \subset C$ which contradicts \eqref{form28a}.  \end{proof} 

\section{Deducing Theorem \ref{thm:ABCConjecture} from Theorem \ref{c3}}\label{s6}

In this section we deduce a slightly weaker version of Theorem \ref{thm:ABCConjecture} from Theorem \ref{c3}. The weaker version is otherwise the same as Theorem \ref{thm:ABCConjecture}, except that the "subset" conclusion \eqref{conclusion2} is replaced by $|A + cB|_{\delta} \geq \delta^{-\chi}|A|$ (in other words, the weaker version only treats the case $G = A \times B$). Modulo reducing the constant "$\chi$", this version of Theorem \ref{thm:ABCConjecture} formally implies the original, stronger version by standard, albeit lengthy, arguments in additive combinatorics, see \cite[Section 5]{MR4778059}.

Quite likely it would be possible -- and easier than going via \cite[Section 5]{MR4778059} -- to deduce Theorem \ref{thm:ABCConjecture} (in the stated strong formulation) directly from Theorem \ref{c3}.

\begin{proof}[Proof of Theorem \ref{thm:ABCConjecture} starting from Theorem \ref{c3}] Let $\alpha,\beta,\gamma$ and $A,B,C \subset \delta \Z \cap [0,1]$ be as in Theorem \ref{thm:ABCConjecture}. We may assume that $\alpha > 0$ and $\beta,\gamma \leq \alpha$, since the other cases are easily reduced to this one. Fix $\bar{\alpha} > \alpha$ and $\eta > 0$ such that $\beta + \gamma \geq \bar{\alpha} + 2\eta$. Let
\begin{equation}\label{def:chi} \chi := (\bar{\alpha} - \alpha) \cdot \min\{\eta/2, \epsilon\} > 0, \end{equation}
where $\epsilon := \epsilon(\bar{\alpha},\beta,\gamma,\eta) > 0$ is the constant given by Theorem \ref{c3} applied with parameters $\bar{\alpha},\beta,\gamma,\eta$. We claim that Theorem \ref{thm:ABCConjecture} holds with constant $\chi$.

 We may assume with no loss of generality that $A$ is $\{2^{-jT}\}_{j = 1}^{m}$-uniform for some $T \geq 1$. Let $f \colon [0,m] \to [0,m]$ be the branching function of $A$. Let $\{a_{j}\}_{j = 0}^{n}$ and $\{\alpha_{j}\}_{j = 1}^{n}$ be the sequences provided by Lemma \ref{l:combinatorial-weak2} applied to $f$. Write
\begin{displaymath} n_{0} := \max\{1 \leq j \leq n : \alpha_{j} \leq \bar{\alpha}\}, \end{displaymath}
and set $\Delta := 2^{-a_{n_{0}}T}$. By the uniformity of $A$, we may write
\begin{displaymath} |A|_{\Delta \to \delta} = |A|_{2^{-a_{n_{0}}T} \to 2^{-a_{n}}T} = \prod_{j = n_{0}}^{n - 1} |A|_{2^{-a_{j}T} \to 2^{-a_{j + 1}T}}, \end{displaymath} 
where the product is empty if $n_{0} = n$. Moreover, by the $(\alpha_{j + 1},0)$-superlinearity of the branching function $f$ on $[a_{j},a_{j + 1}]$ (thanks to Lemma \ref{l:combinatorial-weak2}(2)), the individual factors in the product above satisfy
\begin{displaymath} |A|_{2^{-a_{j}T} \to 2^{-a_{j + 1}T}} \geq 2^{\alpha_{j + 1}T(a_{j + 1} - a_{j})} = \left(\frac{\Delta_{j}}{\Delta_{j + 1}} \right)^{\alpha_{j + 1}}, \end{displaymath}
where $\Delta_{j} = 2^{-a_{j}T}$. Using also that the slopes $\alpha_{j}$ are increasing, 
\begin{displaymath} \delta^{-\alpha} \geq |A| \geq |A|_{\Delta \to \delta} \geq \prod_{j = n_{0}}^{n - 1} \left(\frac{\Delta_{j}}{\Delta_{j + 1}} \right)^{\alpha_{j + 1}} \geq \left(\frac{\Delta}{\delta} \right)^{\bar{\alpha}}, \end{displaymath}
thus $\Delta \leq \delta^{(\bar{\alpha} - \alpha)/\bar{\alpha}} \leq \delta^{\bar{\alpha} - \alpha}$. We claim that $\bar{A} := \mathcal{D}_{\Delta}(A)$ is a Katz-Tao $(\Delta,\bar{\alpha},C)$-set for an absolute constant $C > 0$. To see this, note that the branching function of $\bar{A}$ equals $f|_{[0,a_{n_{0}}]}$. So, by Lemma \ref{lemma3} (applied with $\epsilon = 0$), it suffices to show that
\begin{displaymath} f(a_{n_{0}}) - f(x) \leq \bar{\alpha}(a_{n_{0}} - x), \qquad x \in [0,a_{n_{0}}]. \end{displaymath} 
Fix $x \in [0,a_{n_{0}}]$, and let $j \in \{1,\ldots,n_{0}\}$ such that $x \in [a_{j - 1},a_{j}]$. We claim that $f(a_{j}) - f(x) \leq \alpha_{j}(a_{j} - x)$. Once this has been established, we are done:
\begin{align*} f(n_{0}) - f(x) & = f(a_{j}) - f(x) + \sum_{i = j}^{n_{0} - 1} f(a_{i + 1}) - f(a_{i})\\
& \leq \alpha_{j}(a_{j} - x) + \sum_{i = j}^{n_{0} - 1} \alpha_{i + 1}(a_{i + 1} - a_{i}) \leq \bar{\alpha}(a_{n_{0}} - x). \end{align*}
To prove $f(a_{j}) - f(x) \leq \alpha_{j}(a_{j} - x)$, recall that $f$ is $(\alpha_{j},0)$-superlinear on $[a_{j - 1},a_{j}]$ with $\alpha_{j} = s_{f}(a_{j - 1},a_{j})$. Thus, $f(a_{j}) - f(x) = \alpha_{j}(a_{j} - a_{j - 1}) - (f(x) - f(a_{j - 1})) \leq \alpha_{j}(a_{j} - x)$.

 Next, recall that $B$ is $(\delta,\beta,\delta^{-\chi})$-Frostman and $C$ is $(\delta,\gamma,\delta^{-\chi})$-Frostman. It follows that there exist $\Delta$-separated sets $\bar{B} \subset B$ and $\bar{C} \subset C$ such that $\bar{B}$ is Katz-Tao $(\Delta,\beta)$ and $\bar{C}$ is Katz-Tao $(\Delta,\gamma)$, and
\begin{displaymath} |\bar{B}| \gtrsim \delta^{\chi}\Delta^{-\beta} \quad \text{and} \quad |\bar{C}| \gtrsim \delta^{\chi}\Delta^{-\gamma}. \end{displaymath}
(To see this, note that the $\beta$-dimensional Hausdorff content of $\cup \mathcal{D}_{\Delta}(B)$ exceeds $\delta^{\chi}$, and apply \cite[Proposition 3.9]{OrponenIncidenceGeometry}.) In particular, $|\bar{B}||\bar{C}| \gtrsim \delta^{2\chi}\Delta^{-\beta - \gamma} \geq \Delta^{-\bar{\alpha} - \eta}$ by \eqref{def:chi}, and since $\Delta \leq \delta^{\bar{\alpha} - \alpha}$. Note that $\bar{B},\bar{C}$ are also Katz-Tao $(\Delta,\bar{\alpha})$-sets, since we assumed $\beta,\gamma \leq \alpha \leq \bar{\alpha}$. We have now verified the hypotheses of Theorem \ref{c3} with the triple of exponents $\bar{\alpha},\bar{\alpha},\bar{\alpha}$ instead of $\alpha,\beta,\gamma$, and at scale $\Delta$. It follows that there exists $c \in \bar{C} \subset C$ such that 
\begin{displaymath} |\bar{A} + c\bar{B}|_{\Delta} \geq \Delta^{-\epsilon}|\bar{A}| = \Delta^{-\epsilon}|A|_{\Delta}. \end{displaymath}
Consequently, 
\begin{displaymath} |A + cB|_{\delta} \gtrsim |\bar{A} + c\bar{B}|_{\Delta}|A|_{\Delta \to \delta} \geq \Delta^{-\epsilon}|A| \geq \delta^{-\chi}|A|, \end{displaymath}
and the proof is complete. \end{proof}

\bibliographystyle{plain}
\bibliography{references}

\def\cprime{$'$}
\begin{thebibliography}{10}

\bibitem{Bo1}
Jean Bourgain.
\newblock On the {E}rd\"os-{V}olkmann and {K}atz-{T}ao ring conjectures.
\newblock {\em Geom. Funct. Anal.}, 13(2):334--365, 2003.

\bibitem{Bo2}
Jean Bourgain.
\newblock The discretized sum-product and projection theorems.
\newblock {\em J. Anal. Math.}, 112:193--236, 2010.

\bibitem{MR3361775}
Nicolas de~Saxc\'{e}.
\newblock A product theorem in simple {L}ie groups.
\newblock {\em Geom. Funct. Anal.}, 25(3):915--941, 2015.

\bibitem{2025arXiv251115899D}
Ciprian {Demeter} and William {O'Regan}.
\newblock {Incidence estimates for quasi-product sets and applications}.
\newblock {\em arXiv e-prints}, page arXiv:2511.15899, November 2025.

\bibitem{MR4869897}
Ciprian Demeter and Hong Wang.
\newblock Szemer\'{e}di-{T}rotter bounds for tubes and applications.
\newblock {\em Ars Inven. Anal.}, pages Paper No. 1, 46, 2025.

\bibitem{MR4283564}
Larry Guth, Nets~Hawk Katz, and Joshua Zahl.
\newblock On the discretized sum-product problem.
\newblock {\em Int. Math. Res. Not. IMRN}, (13):9769--9785, 2021.

\bibitem{MR2484645}
Katalin Gyarmati, M\'{a}t\'{e} Matolcsi, and Imre~Z. Ruzsa.
\newblock Pl\"{u}nnecke's inequality for different summands.
\newblock In {\em Building bridges}, volume~19 of {\em Bolyai Soc. Math.
  Stud.}, pages 309--320. Springer, Berlin, 2008.

\bibitem{He19}
Weikun He.
\newblock Discretized sum-product estimates in matrix algebras.
\newblock {\em J. Anal. Math.}, 139(2):637--676, 2019.

\bibitem{MR4148151}
Weikun He.
\newblock Orthogonal projections of discretized sets.
\newblock {\em J. Fractal Geom.}, 7(3):271--317, 2020.

\bibitem{KT01}
Nets~Hawk Katz and Terence Tao.
\newblock Some connections between {F}alconer's distance set conjecture and
  sets of {F}urstenburg type.
\newblock {\em New York J. Math.}, 7:149--187, 2001.

\bibitem{2023arXiv230602943M}
Andr{\'a}s {M{\'a}th{\'e}} and William~Lewis {O'Regan}.
\newblock {Discretised sum-product theorems by Shannon-type inequalities}.
\newblock {\em arXiv e-prints}, page arXiv:2306.02943, June 2023.

\bibitem{2025arXiv250102131O}
William {O'Regan}.
\newblock {Sum-product phenomena for Ahlfors-regular sets}.
\newblock {\em arXiv e-prints}, page arXiv:2501.02131, January 2025.

\bibitem{MR4778059}
Tuomas Orponen.
\newblock On the discretised {$ABC$} sum-product problem.
\newblock {\em Trans. Amer. Math. Soc.}, 377(7):4647--4702, 2024.

\bibitem{OrponenIncidenceGeometry}
Tuomas Orponen.
\newblock Approximate incidence geometry in the plane.
\newblock {\em Res. Math. Sci.}, 12(65), 2025.

\bibitem{OS23}
Tuomas Orponen and Pablo Shmerkin.
\newblock On the {H}ausdorff dimension of {F}urstenberg sets and orthogonal
  projections in the plane.
\newblock {\em Duke Math. J.}, 172(18):3559--3632, 2023.

\bibitem{2023arXiv230110199O}
Tuomas Orponen and Pablo Shmerkin.
\newblock Projections, {F}urstenberg sets, and the {$ABC$} sum-product problem.
\newblock {\em J. Amer. Math. Soc.}, 39(3):857--913, 2026.

\bibitem{MR4912925}
Tuomas Orponen and Guangzeng Yi.
\newblock Large cliques in extremal incidence configurations.
\newblock {\em Rev. Mat. Iberoam.}, 41(4):1431--1460, 2025.

\bibitem{2022arXiv221102277P}
Quy {Pham}, Thang {Pham}, and Chun-Yen {Shen}.
\newblock {Discretized sum-product type problems: Energy variants and
  Applications}.
\newblock {\em arXiv e-prints}, page arXiv:2211.02277, November 2022.

\bibitem{2023arXiv230808819R}
Kevin {Ren} and Hong {Wang}.
\newblock {Furstenberg sets estimate in the plane}.
\newblock {\em arXiv e-prints}, page arXiv:2308.08819, August 2023.

\bibitem{MR2314377}
Imre~Z. Ruzsa.
\newblock An application of graph theory to additive number theory.
\newblock {\em Sci. Ser. A Math. Sci. (N.S.)}, 3:97--109, 1989.

\bibitem{Sh}
Pablo Shmerkin.
\newblock On {F}urstenberg's intersection conjecture, self-similar measures,
  and the {$L^q$} norms of convolutions.
\newblock {\em Ann. of Math. (2)}, 189(2):319--391, 2019.

\bibitem{Shmerkin22}
Pablo Shmerkin.
\newblock A non-linear version of {B}ourgain's projection theorem.
\newblock {\em J. Eur. Math. Soc. (JEMS)}, 25(10):4155--4204, 2023.

\bibitem{SW21}
Pablo Shmerkin and Hong Wang.
\newblock On the distance sets spanned by sets of dimension {$d/2$} in {$\Bbb
  R^d$}.
\newblock {\em Geom. Funct. Anal.}, 35(1):283--358, 2025.

\bibitem{MR729791}
Endre Szemer\'{e}di and William~T. Trotter, Jr.
\newblock Extremal problems in discrete geometry.
\newblock {\em Combinatorica}, 3(3-4):381--392, 1983.

\bibitem{MR2289012}
Terence Tao and Van Vu.
\newblock {\em Additive combinatorics}, volume 105 of {\em Cambridge Studies in
  Advanced Mathematics}.
\newblock Cambridge University Press, Cambridge, 2006.

\bibitem{2024arXiv241108871W}
Hong {Wang} and Shukun {Wu}.
\newblock {Restriction estimates using decoupling theorems and two-ends
  Furstenberg inequalities}.
\newblock {\em arXiv e-prints}, page arXiv:2411.08871, November 2024.

\bibitem{2025arXiv250921869W}
Hong {Wang} and Shukun {Wu}.
\newblock {Two-ends Furstenberg estimates in the plane}.
\newblock {\em arXiv e-prints}, page arXiv:2509.21869, September 2025.

\end{thebibliography}

\end{document}